\documentclass[12pt]{amsart}
\usepackage[margin= 1 in]{geometry}

\usepackage{pdfsync}

 \usepackage[plainpages=false]{hyperref}
 \usepackage{amsfonts,amsmath,amssymb,amsthm,accents,mathtools,amscd, stackrel}
 \usepackage{latexsym,lscape,rawfonts,mathrsfs}

 \usepackage[all]{xy}

 \usepackage{eufrak}
 \usepackage{graphicx,psfrag}
 \usepackage{pstool}

 \usepackage{array,tabularx}

 \usepackage{appendix}

\usepackage{txfonts}

 \newcommand{\ba}{\begin{align}}
 \newcommand{\ea}{\end{align}}
 \newcommand{\bal}{\begin{align*}}
 \newcommand{\eal}{\end{align*}}

 \DeclareMathOperator{\diam}{diam}

 \newcommand{\Rm}{\mathbf{Rm}}
 \newcommand{\Rc}{\mathbf{Rc}}
 \newcommand{\Sc}{\mathbf{R}}
 \newcommand{\K}{\mathbf{K}}

 \newcommand{\dvol}{\text{d}V}

\renewcommand{\epsilon}{\varepsilon}

\newcommand{\rank}{\textbf{rank}}

 \makeatletter
 \def\ExtendSymbol#1#2#3#4#5{\ext@arrow 0099{\arrowfill@#1#2#3}{#4}{#5}}
 
 \makeatother

 \makeatletter
 \def\ExtendSymbol#1#2#3#4#5{\ext@arrow 0099{\arrowfill@#1#2#3}{#4}{#5}}
 
 \makeatother

\newcommand{\leftrarrows}{\mathrel{\raise.75ex\hbox{\oalign{%
  $\scriptstyle\leftarrow$\cr
  \vrule width0pt height.5ex$\hfil\scriptstyle\relbar$\cr}}}}
\newcommand{\lrightarrows}{\mathrel{\raise.75ex\hbox{\oalign{%
  $\scriptstyle\relbar$\hfil\cr
  $\scriptstyle\vrule width0pt height.5ex\smash\rightarrow$\cr}}}}
\newcommand{\Rrelbar}{\mathrel{\raise.75ex\hbox{\oalign{%
  $\scriptstyle\relbar$\cr
  \vrule width0pt height.5ex$\scriptstyle\relbar$}}}}
\newcommand{\longleftrightarrows}{\leftrarrows\joinrel\Rrelbar\joinrel\lrightarrows}

\makeatletter
\def\leftrightarrowsfill@{\arrowfill@\leftrarrows\Rrelbar\lrightarrows}
\newcommand{\xleftrightarrows}[2][]{\ext@arrow 3399\leftrightarrowsfill@{#1}{#2}}
\makeatother

 \definecolor{hao}{rgb}{1,0.5,0}
 \definecolor{miao}{cmyk}{0.5,0,0.2,0.2}
 \definecolor{qiao}{gray}{0.96}

\newtheorem{prop}{Proposition}[section]

\newtheorem{theorem}[prop]{Theorem}

\newtheorem{lemma}[prop]{Lemma}
\newtheorem{claim}[prop]{Claim}

\newtheorem{corollary}[prop]{Corollary}

\newtheorem{definition}[prop]{Definition}

\newtheorem*{theorem*}{Theorem}
\theoremstyle{remark}
\newtheorem{remark}{Remark}

\numberwithin{equation}{section}


\address{Shaosai Huang, Department of Mathematics, University of Wisconsin -
Madison, 480 Lincoln Drive, Madison, WI, 53706, U.S.A.}
\email{sshuang@math.wisc.edu}

\address{Bing Wang, Institute of Geometry and Physics, and School of
Mathematical Sciences, University of Science and Technology of China, 96
Jinzhai Road, Hefei, Anhui Province, 230026, China}
\email{topspin@ustc.edu.cn}

\title[Local smoothing by Ricci flows]{Ricci flow smoothing for
locally collapsing manifolds}
\author{Shaosai Huang}
\author{Bing Wang}
\date{\today}

\begin{document}
\maketitle

\begin{abstract}
We show that for certain locally collapsing initial data with
Ricci curvature bounded below, one could start the Ricci flow for a
definite period of time. This provides a Ricci flow smoothing tool, with which
we find topological conditions that detect the collapsing infranil fiber
bundles over controlled Riemannian orbifolds among those locally collapsing
regions with Ricci curvature bounded below. In the appendix, we also provide a
local distance distortion estimate for certain Ricci flows with collapsing
initial data.
\end{abstract}
\section{Introduction}

The Ricci flow with initial data $(M,g)$, first introduced by Hamilton
\cite{Hamilton} on $3$-manifolds to deform a given Riemannian metric $g$ with
positive Ricci curvature to a canonical one, is a smooth family of Riemannian
metrics $g(t)$ on $M$ solving the following initial value problem for $t\ge 0$: 
\begin{align}\label{eqn: RF_defn}
\begin{cases}
\partial_tg(t)\ &=\ -2\Rc_{g(t)};\\
g(0)\ &=\ g.
\end{cases}
\end{align} 
Hamilton \cite{Hamilton} shows that if $M$ is a closed manifold, this initial
value problem is always solvable up to a certain time, depending on the initial
data $(M,g)$. In harmonic coordinates, the Ricci flow becomes a non-linear
heat-type equation for the metric tensor, and by the nature of the heat flows,
notably Shi's estimates \cite{Shi}, a key effect of running Ricci flow is that
the evolved metric has much improved regularity:
\begin{align}\label{eqn: Shi_estimate}
\forall l\in \mathbb{N},\ \exists C_l>0,\ \sup_M\left|\nabla^l
\Rm_{g(t)}\right|_{g(t)}\ \le\ C_lt^{-1-l}.
\end{align}
Here the constants $C_l$ depend on the dimension of $M$, as well as 
$\left\|\Rm_g\right\|_{L^{\infty}(M,g)}$. In fact, the finiteness of
$\left\|\Rm_g\right\|_{L^{\infty}(M,g)}$ guarantees the Ricci flow solution to
(\ref{eqn: RF_defn}) to exist for a definite amount of time determined by
its value, even if $(M,g)$ is complete but non-compact.

In view of Shi's estimates (\ref{eqn: Shi_estimate}), the Ricci flow also
becomes a useful tool to smooth a given Riemannian metric by replacing the
initially given metric $g$ with the evolved metric $g(t)$, whose regularity is
controlled by (\ref{eqn: Shi_estimate}). In order to uniformly control the
evolved metric according to (\ref{eqn: Shi_estimate}), a uniform lower bound
of the Ricci flow existence time then becomes crucial.

When the initial data has bounded sectional curvature, notable applications of
the Ricci flow smoothing method include \cite{Fukaya88, Fukaya89, CFG92}, where
the regularity of the smoothing metric plays a key role in understanding the
fine structures of the collapsing geometry with bounded curvature. Imposing
uniform two-sided bounds on the sectional curvature is a rather strong
requirement, and a more general (and natural) situation is to only assume a
uniform Ricci curvature lower bound. In this case, the Ricci flow smoothing
technique usually enables one to obtain Cheeger-Gromov (smooth) convergence to
the regular part of the \emph{non-collapsing} Gromov-Hausdorff (rough) limit
spaces. Applications of such technique have been illustrated by the work
\cite{TianWang} on almost Einstein manifolds, by the work \cite{LS18} on the
Gromov-Hausdorff limits of K\"ahler manifolds with Ricci curvature bounded
below, and by the work \cite{ST17} where the Gromov-Hausdorff limits of
$3$-manifolds with Ricci curvature bounded below is shown to be topological
manifolds, confirming the $3$-dimensional case of a conjecture due to Cheeger
and Colding \cite{ChCoI}.

In all the above mentioned results, the existence time lower bound of Ricci
flows --- which is critical for the smoothing purpose as we have discussed --- 
depends on the uniform volume ratio lower bound of the initial metric:
when the initial volume ratio at some point becomes smaller, the existence time of
the Ricci flow becomes shorter. However, in many natural situations, especially
for the purpose of smoothing a given initial metric by running the Ricci flow,
there may be no uniform volume ratio lower bound to be assumed; and the
purpose of the current paper, sequential to the previous work \cite{Foxy1808},
is then to show that in certain cases when the initial data have a uniform
Ricci curvature lower bound but \emph{without} any uniform volume ratio lower
bound, the Ricci flow could still be started locally for a definite period of
time. Throughout the paper, for a compact subset $K\subset M$ and $R>0$, we
will let $B_g(K,R)$ denote the geodesic $R$-neighborhood of $K$, i.e. 
$B_g(K,R)\ :=\ \left\{x\in M:\ d_g(x,K)<R\right\}$. Our first result concerns
starting Ricci flows with possibly collapsing initial data locally modeled on
Euclidean spaces:
\begin{theorem}\label{thm: main1}
For any $\alpha\in (0,10^{-1})$ and $R\in (0,100)$, there are positive constants
$\delta_{E}(m,R,\alpha)<1$ and $\varepsilon_{E}(m,R,\alpha)<1$ to the following
effect: let $K$ be a compact and connected subset of $(M^m,g)$, an
$m$-dimensional complete Riemannian manifold with $\Rc_g\ge -(m-1)g$, if for
some $\delta\le\delta_{E}$ and any $p\in B_g(K,R)$ it satisfies
\begin{enumerate}
  \item $d_{GH}\left(B(p,10^{-1}R),\mathbb{B}^k(10^{-1}R)\right)<\delta$, and
  \item $\rank\ \tilde{\Gamma}_{\delta}(p)=m-k$,
\end{enumerate}
then there is a Ricci flow solution with initial data
$(B_g(K,\frac{R}{4}),g)$, existing for a period no shorter than
$\varepsilon_{E}^2$, and with curvature control
\begin{align}
\forall t\in (0,\varepsilon_{E}^2],\quad
\sup_{B_g(K,\frac{R}{4})}\left|\Rm_{g(t)}\right|_{g(t)}\ \le\ \alpha t^{-1}
+\varepsilon_{E}^{-2}.
\end{align}
\end{theorem}

Here $\tilde{\Gamma}_{\delta}(p)$ denotes the
\emph{pseudo-local fundamental group}, and is defined for each $\delta\in (0,1)$
as
\begin{align*}
\tilde{\Gamma}_{\delta}(p)\ :=\ Image[\pi_1(B_g(p,\delta),p)\to
\pi_1(B_g(K,R),p)].
\end{align*}
This concept is introduced in \cite{HW20a} and is originated from the concept
of the \emph{fibered fundamental group} $\Gamma(p):=
Image[\pi_1(B_g(p,\delta),p) \to \pi_1(B_g(p,2),p)]$, which captures all those
loops based at $p$ and contained in $B_g(p,\delta)$, but are allowed to be
deformed (with fixed base point) within $B_g(p,2)$. The fibered fundamental
group has been a crucial concept for our understanding of the local structure
of manifolds with Ricci curvature bounded below. The important work of Kapovitch
and Wilking \cite{KW11} has shown that the fibered fundamental group is almost
nilpotent, and so is the pseudo-local fundamental group, according to the work
of Naber and Zhang \cite{NaberZhang}: see also \cite[Lemma 2.2]{HW20a}. In fact,
the alomst nilpotency of groups of these kinds has been a key property to
investigate even for manifolds with sectional curvature bounded below; see also
the previous works \cite{FY92, KPT10} for some remarkable results.


Our previous result \cite[Theorem 1.4]{HW20a} tells that if a closed
manifold $(M,g)$ with $\Rc_{g}\ge -(m-1)g$ is sufficiently
Gromov-Hausdorff close to some closed manifold $(N,h)$ with
bounded geometry, if $b_1(M)-b_1(N)=\dim M-\dim N$, then the Ricci flow
with initial data $(M,g)$ exists for a definite amount of time, independent of
the volume $|M|_g$. This theorem is global in nature, and sees limited
applications in general settings --- similar issues arise for the work
\cite{DWY96}, where the Ricci flow smoothing is applied to closed manifolds
with bounded Ricci curvature for the first time. 

In contrast, an important feature of Theorem~\ref{thm: main1} is that it is
purely local. Since the assumed Ricci curvature lower bound could be
directly obtained via rescaling, this theorem lends itself as an agile
smoothing tool in various contexts. Also notice that the Assumption (1) in
Theorem~\ref{thm: main1} can be replaced as $B_g(K,R)$ is
$\delta$-Gromov-Hausdorff close to a given lower dimensional smooth manifold
(not necessarily complete) with a uniform lower bound on the
$C^{1,\frac{1}{2}}$ harmonic radius. Moreover, this theorem can be generalized
to the setting where the initial data may locally collapse with only
\emph{scalar curvature} bounded below: see Theorem~\ref{thm: sc}.

If the collapsing limit is a controlled Riemannian orbifold, there are still
certain conditions that guarantee the Ricci flow to exist for a definite period of
time. Our second result concerns starting the Ricci flow on possibly collapsing
initial data locally modeled on flat orbifolds:
 \begin{theorem}\label{thm: main2}
For any $\alpha\in (0,10^{-1})$, $m, l\in \mathbb{N}$ and $R\in (0,100)$, there
are positive constants $\delta_{O}(m,l,R,\alpha)<1$ and
$\varepsilon_{O}(m,\alpha)<1$ to the following effect: let $K$ be a compact and
connected subset of $(M^m,g)$, an $m$-dimensional Riemannian manifold with
$\Rc_g\ge -(m-1)g$, suppose for some $k\le m$ and $\delta\le \delta_O$ it
satisfies for any $p\in B_g(K,R)$ the following assumptions:
\begin{enumerate}
   \item  there are a finite group $G_p<O(k)$ with $\left|G_p\right|\le l$ and a
surjective group homomorphism $\phi_p:\pi_1(B_g(K,R),p)\twoheadrightarrow G_p$,
   \item $d_{GH}\left( B_g(p,4^{-1}R), \mathbb{B}^k(4^{-1}R)\slash
   G_p\right)<\delta$, and
    \item $\rank\ \tilde{\Gamma}_{\delta}(p)=m-k$,
 \end{enumerate}
then there is a Ricci flow solution with initial data $(B_g(K,\frac{R}{4}),g)$,
existing for a period no shorter than $\varepsilon_{O}^2$, and with curvature control 
\begin{align}
\forall t\in (0,\varepsilon_{O}^2],\quad
\sup_{B_g(K,\frac{R}{4})}\left|\Rm_{g(t)}\right|_{g(t)}\ \le\ \alpha t^{-1}
+\varepsilon_{O}^{-2}.
\end{align}
\end{theorem}

Again, here we may say that $B_g(K,R)$ is $\delta$-Gromov-Hausdorff
close to a $k$-dimensional controlled orbifold, in the sense that each
point has its orbifold group of size bounded by $l$ (regular points are thought
of as having the unit-size orbifold group), and that the orbifold covering
metric has a uniform $C^{1,\frac{1}{2}}$ harmonic radius lower bound.
\begin{remark}\label{rmk: general_orbifold}
Notice that the concept of obifold involved here is more general than the usual
definition (see e.g. \cite{BKN89}) in that the action of $G<O(k)$ may
\emph{not} be free on $\mathbb{S}^{k-1}$. By \cite[Theorem 0.5]{Fukaya88}, such
orbifold singularities may occure as collapsing limits. This has become the
major difficulty in our proof of the theorem, which relies on the recent
developments \cite{DPG18} in the theory of $RCD$ spaces introduced by Ambrosio,
Gigli and Savar\'e \cite{AGS14}: see Lemma~\ref{lem: isometry}.
\end{remark}

We also notice that if $l=1$, then Assumption (1) of this theorem is
automatically satisfied --- we are reduced to the situation considered in
Theorem~\ref{thm: main1}. While this theorem is more general than the case of
initial data locally collapsing to the Euclidean model space, we will still
begin with providing a detailed proof of Theorem~\ref{thm: main1} in the next
section. We then discuss the necessary alternations leading to a proof of
Theorem~\ref{thm: main2} in the section that follows. The final section
contains an application of our Ricci flow smoothing results: we will detect the
infranil fiber bundles over controlled Riemannian orbifolds among all those
collapsing manifolds with Ricci curvature bounded below. Here we specify the
concept of controlled orbifold in the following
\begin{definition}
 We say that a metric space $(Z,d_Z)$ is a locally
 $(l,\delta,\bar{\iota})$-controlled $k$-dimensional Riemannian orbifold if
\begin{enumerate}
  \item[(a).] $(Z,d_Z)$ is a $k$-dimensional Riemannian orbifold;
\item[(b).] $\forall z\in Z$, there is a finite group $G_z<O(k)$ of order
not exceeding $l$, such that 
\begin{align*}
d_{GH}\left(B_{d_Z}(z,\bar{\iota}), \mathbb{B}^k(\bar{\iota})\slash
G_z\right)\ <\ \delta.
\end{align*} 
\end{enumerate}
\end{definition} Again, here we notice
that each $G_z<O(k)$ may well have a non-empty fixed point set in
$\mathbb{S}^{k-1}$.
With such definition, we have
\begin{theorem}\label{thm: main3}
Given $m,l\in \mathbb{N}$ and $\bar{\iota}\in (0,1)$, there is a positive
constant $\delta_F(m,l,\bar{\iota})>0$ to the following effect: let $K$ be a
compact and connected subset of $(M^m,g)$, an $m$-dimensional Riemannian
manifold with $\Rc_g\ge -(m-1)g$, suppose for some $k\le m$ and $\delta\le
\delta_F$ it satisfies the following assumptions:
\begin{enumerate}
   \item  there is a locally $(l,\delta,\bar{\iota})$-controlled $k$-dimensional
   Riemannian orbifold $(Z^k,d_Z)$ such that $d_{GH}\left(B_g(K,4\bar{\iota}),
   Z\right)<\delta_F(m,l,\bar{\iota})$, and
   \item with $\Phi: B_g(K,4\bar{\iota})\to Z$ denoting a
   $\delta$-Gromov-Hausdorff approximation, 
   \begin{align*}
   \forall\ p\in B_g(K,R),\quad \rank\
   \tilde{\Gamma}_{\delta}(p)=m-k,\quad\text{and}\quad \exists\ \phi_p\in
   Hom\left(\pi_1(B_g(K,R),p),G_{\Phi(p)}\right) \end{align*}
   which is surjective, where $G_{\Phi(p)}$ is the orbifold group at
   $\Phi(p)\in Z$,
 \end{enumerate}
then there is an open subset $U\subset M$ with $K\Subset U \Subset
B_{g}(K,\bar{\iota})$, such that $U$ is an infranil fiber bundle over 
 $Z_{4\bar{\iota}}:=\left\{z\in Z:\ d_Z(z,\partial Z)>
4\bar{\iota}\right\}$.
\end{theorem}
The emphasis of this theorem is of course the collapsing setting, i.e. the case 
$k<m$. When $k=m$ it can be seen that there cannot be a surjective group
homomorphism from the fundamental group to the orbifold group. Since otherwise,
there would be a sequence of $m$-dimensional Riemannian manifolds with
uniformly bounded curvature converging to an $m$-dimensional orbifold in the
pointed Gromov-Hausdorff sense, and this is impossible.

Theorem~\ref{thm: main3} generalizes \cite[Proposition 6.6]{NaberZhang} in two
fronts: it removes the Ricci curvature upper bound assumption when
$\left|G_{\Phi(z)} \right|=1$, and it allows the collapsing limit to be any
singular orbifold when $\left|G_{\Phi(z)}\right|>1$. While the removal of the
assumed Ricci curvature upper bound in \cite[Proposition 6.6]{NaberZhang} is
expected among experts (see \cite[Remark 6.5]{NaberZhang} and
Remark~\ref{rmk: HR20}), perhaps the more substantial contribution of
Theorem~\ref{thm: main3} is that it copes with singular collapsing limits (see 
Remark~\ref{rmk: Reifenberg}), giving a topological condition that detects the
nilpotent Killing structure \emph{a la} Cheeger, Fukaya and Gromov
\cite[Definition 1.5]{CFG92} in the context of Ricci curvature bounded from
below.
 
From the perspective of applications, Theorem~\ref{thm: main3} will provide a
technical tool for the possible generalization of our rigidity theorem on the
first Betti number (\cite[Theorem 1.1]{HW20a}) to the situation of controlled 
singular orbifolds as collapsing limits. See also \cite[Remarks 0.4 -
0.6]{Huang20} and the third last paragraph on \cite[Page 2]{HKRX18} 
for potential applications of Theorem~\ref{thm: main3} in the study of the
locally bounded Ricci covering geometry (see e.g. \cite{HRW20, Rong20}) ---
notice that the maximal nilpotency rank assumption on the fibered fundamental
groups guarantees the local universal covering to be uniformly non-collapsing,
as shown in \cite{NaberZhang}.
 

\section{Initial data locally collapsing to Euclidean model spaces} 
In this section we prove our first Ricci flow local existence theorem.
We fix a complete Riemannian manifold $(M,g)$ with
 Ricci curvature bounded below as $\Rc_{g}\ge -(m-1)g$. Given a compact
subset $K\subset M$, the neighborhood $B_g(K,R):=\{x\in M:\
d_g(x,K)<R\}$ has compact closure in $M$ for any $R>0$, thanks to the
completeness of $(M,g)$.

Fixing any $R>0$ and any compact subset $K\subset M$, we could already 
start a Ricci flow with initial data $(B_g(K,R),g)$, based on the conformal 
transformation technique due to Hochard \cite[\S 6]{Hochard} (see also
\cite{ST17, He16, Lai19} for other applications and refinements of
this technique) --- we will blow the boundary to infinity using a well-known
conformal factor, and apply Shi's short-time existence theorem to the newly
obtained metric:
 \begin{theorem}[Shi's short-time existence]\label{thm: Shi}
There are positive constants $C_S(m)$ and $T_S(m,K)$ such that if
$(M^m,g)$ is an $m$-dimensional complete Riemannian manifold with sectional
curvature uniformly bounded by $K>0$ in absolute value, then there exists a
complete Ricci flow solution $g(t)$ defined on $M\times [0,T_S]$, satisfying 
\begin{align}\label{eqn: Shi_Rm}
\forall t\in [0,T_S],\quad \sup_{M}\left|\Rm_{g(t)}\right|_{g(t)}\ \le\
C_St^{-1}.
\end{align}
\end{theorem}
 We now prove the following short-time existence result: 
 \begin{lemma}[Locally starting the Ricci flow]\label{lem: start_RF}
 Let $K\subset M$ be a compact subset of a complete Riemannian manifold
 $(M^m,g)$ with $\Rc_g\ge -(m-1)g$. For any $R>0$ there is a smooth family of
 Riemannian metrics $g(t)$ on $B_g(K,\frac{R}{4})$ satisfying
 \begin{align*}
 \begin{cases}
 \partial_tg(t)\ &=\ -2\Rc_{g(t)}\quad \text{on}\ B_g(K,\frac{R}{4})\times
 [0,T],\\
  g(0)\ &=\ g\quad \text{on}\ B_g(K,\frac{R}{4}),
 \end{cases}
 \end{align*}
 defined up to some time $T>0$ such that
 \begin{align*}
 \forall t\in (0,T],\quad  \sup_{B_g(K,\frac{R}{4})}
 \left|\Rm_{g(t)}\right|_{g(t)}\ \le\ Ct^{-1},
 \end{align*}
 where the positive constants $C$ and $T$ depend on $g$, $K$ and $R$.
  \end{lemma}
 \begin{proof}
Let $\{x_j\}\subset B_g(K,\frac{R}{4})$ be a maximal collection of points such
that $B_g(x_j,\frac{R}{8})\cap B_g(x_{j'},\frac{R}{8})=\emptyset$ whenever
$j\not= j'$. It is clear that $\left\{B_g(x_j,\frac{R}{4})\right\}$ covers
$B_g(K,\frac{R}{4})$ and the multiplicity of this covering is bounded above, at
each point of $B_g(K,\frac{R}{4})$, by
$V_{-1}^m(\frac{5R}{8})V_{-1}^m(\frac{R}{8})^{-1}$, where 
$V_{-1}^m(r)$ denotes, for any $r>0$, the volume of a geodesic $r$-ball in the
space form of sectional curvature equal to $-1$.

We now let $\psi_j:M\to [0,1]$ be the cut-off function constructed in
\cite{ChCo0}, supported in $B_g(x_j,\frac{R}{2})$ with  
$\psi_j|_{B_g(x_j,\frac{R}{4})}\equiv 1$, and satisfying the control $R\|\nabla
\psi_j\|_{C^0(M)}+R^2\|\Delta_g\psi_j\|_{C^0(M)}\le C(m)$. We then have
$\psi:=\sum_j\psi_j$ supported in $B_g(K,R)$, $\psi(x)\ge 1$
whenever $x\in B_g(K,\frac{R}{4})$, and
\begin{align*} 
R\|\nabla\psi\|_{C^0(M)} +R^2\|\Delta_g\psi\|_{C^0(M)}\ \le\
\frac{C(m)V_{-1}^m(\frac{5R}{8})}{V_{-1}(\frac{R}{8})}\ =:\ C_0(m,R).
\end{align*}

 We could find a smooth cut-off function
$u:[0,\infty)\to [0,1]$, such that $u|_{[1,\infty)}\equiv 0$, $u(0)=1$,
$-10<u'<0$ and $|u''|<10$.
We then define the function 
$\rho:=u\circ \psi: B_g(K,\frac{3R}{4})\to [0,1]$, which clearly satisfies the
gradient bound
\begin{align}\label{eqn: rho_C2}
R\|\nabla \rho\|_{C^0(B_g(K,\frac{3R}{4}))}
+R^2\|\Delta_g\rho\|_{C^0(B_g(K,\frac{3R}{4}))}\ \le\ 100C_0(m,R).
\end{align}

For any $\theta \in (0,\frac{1}{2})$ to be fixed later, we define 
the function $w_{\theta}:[0,1)\to [0,\infty)$ as
\begin{align}\label{eqn: f_defn}
w_{\theta}(s)\ :=\ \begin{cases}
\ 0\quad &\text{for}\ s\in [0,1-\theta];\\
\ -\ln (1-(s-1+\theta)^2\theta^{-2})\quad &\text{for}\ s\in [1-\theta,1).
\end{cases}
\end{align}
Notice that for each $s \in (1-\frac{3}{2}\theta,1)$, with 
$\zeta=\zeta(\theta,s):=\frac{1}{4}\theta(1-s)>0$, we have
 \begin{align}\label{eqn: seeking_beta}
0 < s-\zeta < s+\zeta < 1,\quad e^{w_{\theta}(s+\zeta)-w_{\theta}(s-\zeta)}
\le 1+2\theta,\quad \text{and}\quad 10^{-2}\theta^2 \le \zeta
e^{w_{\theta}(s-\zeta)} \le 1.
 \end{align}
In fact, we could perturb $w_{\theta}$ slightly to make it smooth,
still vanishing identically on $[0,1-\frac{3}{2}\theta)$, while keeping
(\ref{eqn: seeking_beta}) and the following derivative control true 
for $s\in [1-\frac{3}{2}\theta,1)$:
\begin{align}\label{eqn: tf_C2}
0\ <w_{\theta}'(s)\ \le\ \frac{2\theta}{\theta^2-(s-1+\theta)^2}\quad 
\text{and}\quad 0\ <\ w_{\theta}''(s)\ \le\ 
\frac{4\theta^2}{(\theta^2-(s-1+\theta)^2)^2}.
\end{align}

Now we define the conformal factor $f_{\theta}=w_{\theta}\circ
\rho:B_g(K,\frac{3R}{4})\to [0,\infty)$, and consider the new metric
$h_{\theta}=e^{2f_{\theta}}g$, defined on $B_g(K,\frac{3R}{4})$. The
function $f_{\theta}$ blows to infinity the boundary points of $\partial
B_g(K,\frac{3R}{4})$, and $h$ becomes a complete Riemannian metric on
$B_g(K,\frac{3R}{4})$.
Since $f_{\theta}(x)=0$ whenever $x\in B_g(K,\frac{R}{4})$, we have
$h_{\theta}\equiv g$ on $B_g(K,\frac{R}{4})$, which contains $K$. Moreover,
since $\overline{B_g(K,\frac{3R}{4})}$ is compact in the complete Riemannian
manifold $M$, we have
\begin{align*}
\sup_{\wedge^2TB_g(K,\frac{3R}{4})}|\K_g|\ =:\ \kappa(K,g,R)\ <\ \infty,
\end{align*}
with $\K_g$ denoting the sectional curvature evaluated at some tangent plane,
viewed as a point in $\wedge^2TB_g(K,\frac{3R}{4})$. We could now compute
the corresponding sectional curvature of $h$ under an orthonormal frame
$\{e_a\}$ as following:
\begin{align*}
(\K_{h_{\theta}})_{ab}\ =\ e^{-2f_{\theta}}\left((\K_g)_{ab}-\sum_{c\not=
a,b}|\nabla_cf_{\theta}|^2 +\nabla^2_{aa}f_{\theta}+\nabla^2_{bb}f_{\theta} 
\right).
\end{align*}
By the estimates (\ref{eqn: rho_C2}) on $\rho$ and (\ref{eqn: tf_C2}) on
$u_{\theta}$, we have the bounds 
\begin{align}\label{eqn: f_C2}
\begin{split}
e^{-f_{\theta}}\|f_{\theta}\|_{C^1(B_g(K,\frac{3R}{4}))}\ &\le\
10^7kC_{0}(m,R)^2\theta^{-1}R^{-2},\\
e^{-2f_{\theta}}\|f_{\theta}\|_{C^2(B_g(K,\frac{3R}{4}))}\ &\le\
10^9k^2C_{0}(m,R)^2\theta^{-2}R^{-2} +10^3k\|\psi\|_{C^2(M)}
\theta^{-3}\ <\ \infty,
\end{split}
\end{align} 
 and since $\kappa(g)<\infty$, while $f_{\theta}(x)\to
\infty$ as $x\to \partial B_g(K,\frac{3R}{4})$, we see that 
\begin{align}\label{eqn: kappa_C2}
\sup_{\wedge^2TB_g(K,\frac{3R}{4})}|\mathbf{K}_h|\ \le\
10^9k^2\left(C_0(m,R)^2R^{-2}+\kappa(K,g,R)
+\|\psi\|_{C^2(M)}\right) \theta^{-4}\ <\ \infty.
\end{align}
We could therefore appeal to Shi's existence theorem (Theorem~\ref{thm: Shi})
to start a Ricci flow $h(t)$ on the complete non-compact Riemannian manifold
$(B_g(K,\frac{3R}{4}),h)$, which has a global curvature bound depending on
$\kappa(K,g,R)$, $\|\psi\|_{C^2(M)}$ and $\theta$. The maximal
existence time of the Ricci flow $h(t)$ is therefore bounded below by some
$T>0$ determined by $m, k,\theta,\kappa(K,g,R)$ and $\|\psi\|_{C^2(M)}$.
Shi's existence theorem also provides the following curvature bound:
\begin{align*}
\forall t\in (0,T],\quad
\sup_{B_g(K,\frac{R}{4})}\left|\Rm_{h(t)}\right|_{h(t)}\ \le\ C_St^{-1},
\end{align*} 
where $C_S>0$ depends on $m,k,\kappa(g)$ and
$\|\psi\|_{C^2(M)}$. Now restricting the flow $h(t)$ to
$K$, where $g\equiv h=h(0)$, we obtain the desired Ricci flow
with initial data $(K,g)$ and curvature bound.
 \end{proof}
 
 \begin{remark}
 In fact, we can also let $\frac{9R^2}{16}(1-\psi)$ be obtained by slightly
 smoothing the square of the distance to $B_g(K,\frac{R}{4})$. The key
 Laplacian upper bound of this cut-off function is then a consequence of the
 Laplacian comparison: $\Delta_gd^2\le 2n$. Such cut-off function has been
 discussed in \cite{YauLecture}.
 \end{remark}
 
Although this lemma enables us to start a Ricci flow locally, we have no
uniform control on the flow, as both the maximal existence time and the
curvature bound in (\ref{eqn: Shi_Rm}) depend on the specific geometric
structure of the space in consideration, encoded in $\kappa(K,g,R)$ and
$\|\psi\|_{C^2(M)}$ as shown by (\ref{eqn: kappa_C2}).
To prove Theorem~\ref{thm: main1}, we will need to reduce their
depence to only on the Ricci curvature. 

A typical approach to obtain a uniform lower bound on the maximal existence
time of a Ricci flow is to invoke Perelman's pseudo-locality theorem (see
\cite[Theorem 10.1]{Perelman}, or another version \cite[Proposition
3.1]{TianWang}). In the complete non-compact setting, a detailed proof of the
pseudo-locality theorem could be found in \cite[Theorem 8.1]{CTY11}. We will
let $I_g(U)$ denote the isoperimetric constant of the domain $U$ equipped with
the metric $g$, while letting $I_m=m^m\omega_m$ denote the $m$-dimensional
Euclidean isoperimetric constant, and state the pseudo-locality theorem as the
following:
\begin{theorem}[Perelman's pseudo-locality for Ricci flows]
\label{thm: CTY11} 
For any $\alpha\in (0,10^{-1})$, there are positive constants
$\delta_{P}(m,\alpha)<1$ and $\varepsilon_P(m,\alpha)<1$, such that for any
$m$-dimensional complete Ricci flow solution $(M,g(t))$ defined on $t\in
[0,T)$, if each time slice has bounded curvature, then for any $x\in M$
satisfying
\begin{align}
 \inf_{B_{g(0)}(x,1)} \Sc_{g(0)}\ &\ge\ -1,
\label{eqn: Perelman_Sc_lb}\\
\text{and}\quad I_{g(0)}\left(B_{g(0)}(x,1)\right)\ &\ge\
(1-\delta_{P})I_m,
\label{eqn: Perelman_isoperimetric}
\end{align}
we have the following curvature bound for any $t\in [0,T)\cap
(0,\epsilon_P^2]$:
\begin{align}\label{eqn: pseudolocality_Rm_bound}
 \left|\Rm_{g(t)}\right|_{g(t)}(x_0)\ \le\ \alpha t^{-1}+\varepsilon_P^{-2}.
\end{align}
\end{theorem}
For those initial data satisfying the assumptions of Theorem~\ref{thm:
CTY11} at every point, we could obtain the uniform existence time lower bound
by a contradiction argument: if the existence time $T$ of the Ricci
flow is shorter than $\varepsilon_P^2$, then for some sequence
$t_i\nearrow T$ we could observe points $x_i\in M$ such that
$\lim_{t_i\to T}\left|\Rm_{g(t_i)}\right|_{g(t_i)}(x_i)= \infty$; especially,
we will get $\left|\Rm_{g(t_i)}\right|_{g(t_i)}(x_i)>2\alpha
T^{-1}+\varepsilon_P^{-2}$ for all $i$ large enough, contradicting the
conclusion (\ref{eqn: pseudolocality_Rm_bound}) since $T>0$ is fixed.

In the setting of Theorem~\ref{thm: main1}, however, we could not
directly apply the pseudo-locality theorem to the Ricci flow obtained from 
Lemma~\ref{lem: start_RF}, since the almost Euclidean isoperimetric constant
assumption (\ref{eqn: Perelman_isoperimetric}) fails drastically for
the initial data in our consideration. In order to overcome this difficulty, we
will pull the conformally transformed initial metric back to the local
universal covering space, which, under the maximal rank assumption of
Theorem~\ref{thm: main1} (Item (2)), is non-collapsing (see
\cite[Proposition 5.9]{NaberZhang}). By the known connection between the
isoperimetric constant and the volume ratio lower bound in the setting of Ricci
curvature bounded below (see e.g. \cite{Gromov}),
we then expect to improve the isoperimetric constant lower bound on the
covering space. We begin with the following
\begin{lemma}[Almost Euclidean condition for local normal covering spaces]
\label{lem: almost_Euclidean}
For any sufficiently small $\varepsilon>0$ fixed, there are positive
constants $\delta_{AE}\le 1$ and $r_{AE}\le 1$, solely determined by
$\varepsilon$ and $m$, to the following effect: let
$B_g(p,10)$ be a geodesic ball in a complete Riemannian manifold $(M^m,g)$
with $\Rc_g\ge -(m-1)g$, let $\pi: Y\to B_g(p,10)$ be any normal
covering with deck transformation group $G$; suppose for some
$\delta\le \delta_{AE}$ it holds
\begin{enumerate}
  \item $d_{GH}\left(B_g(p,10),\mathbb{B}^k(10)\right)<\delta $, and
  \item the group $\widehat{G}_{\delta}(p):=\left\langle \gamma\in G:\
  d_{\pi^{\ast}g}(\gamma.\tilde{p},\tilde{p})<2\delta\right\rangle$ has
  nilpotency rank equal to $m-k$,
 \end{enumerate}
then for any $r\in (0,r_{AE}]$ and any $\tilde{x}\in \pi^{-1}(B_g(p,7))\subset
Y$ we have
\begin{align}
\left|B_{\pi^{\ast}g}(\tilde{x},r)\right|_{\pi^{\ast}g}\ &\ge\
(1-\varepsilon)\omega_mr^m,\label{eqn: almost_max_vol}\\
\quad \text{and}\quad I_{\pi^{\ast}g}\left(B_{\pi^{\ast}g}(\tilde{x},r)\right)\
&\ge\ (1-\varepsilon)I_m. \label{eqn: AE_isoperimetric}
\end{align}
\end{lemma}
 
\begin{remark}\label{rmk: nonlocal_rank}
We will rely on \cite[Lemma 5.3 (ii)]{NaberZhang}, the ``non-localness''
property of the nilpotency rank, whose proof is based on
\cite[Theorem 2.26]{NaberZhang} and \cite[Lemma 5.2]{NaberZhang}. Here we
notice that \cite[Lemma 5.2]{NaberZhang} applies to any discrete isometric
group action, and if we replace \cite[Theorem 2.26]{NaberZhang} by
\cite[Theorem 4.25]{NaberZhang}, then the conclusion of \cite[Lemma 5.3
(ii)]{NaberZhang} holds for any normal covering: if $\pi: X\to B_g(p,10)$ is a
normal covering with deck transformation group $G$, then for any
$\varepsilon>0$ sufficiently small, there is a constant
$\Psi_{NZ}(\varepsilon|m)\in (0,\varepsilon)$, such that
\begin{align}\label{eqn: nonlocal_rank}
\forall x\in B_g(p,7),\quad \rank\ \widehat{G}_{\varepsilon}(x)\
\ge\ \rank\ \widehat{G}_{\Psi_{NZ}(\varepsilon|m)}(p),
\end{align}
where we recall that $\widehat{G}_{\delta}(x)=\left\langle \gamma\in G:\
d_{\pi^{\ast}g}(\gamma.\tilde{x},\tilde{x})<2\delta\right\rangle$, for any
$\tilde{x}\in \pi^{-1}(x)\subset X$.
\end{remark}

\begin{proof}
By \cite[Theorem 1.1]{CM18}, we obtain dimensional constants $C_{0,m}>0$, 
$\bar{\delta}_{0,m}>0$, $\bar{\epsilon}_{0,m}>0$ and $\bar{\eta}_{0,m}>0$. We
require that $\varepsilon \le
\min\{\bar{\delta}_{0,m},\bar{\varepsilon}_{0,m},\bar{\eta}_{0,m},10^{-1}\}$ and 
put $\varepsilon':=\frac{\varepsilon}{16} \min\left\{1,C_{0,m}^{-1}\right\}$ for
all such $\varepsilon$. We also let $r_0=r_0(\varepsilon')\in (0,1)$ be the
constant such that
\begin{align}\label{eqn: choosing_r_0}
\forall r\in (0,r_0],\quad (1-\varepsilon')\omega_mr^m\ \le\
V_{-1}^m(r)\ \le\ (1+\varepsilon')\omega_mr^m,
\end{align}
where $V_{-1}^m(r)$ is the volume of geodesic $r$-ball in the space form of
sectional curvature equal to $-1$. 

By Colding's volume continuity theorem, \cite[Main Lemma 2.1]{Colding97}, we
obtain for $\varepsilon'$ the corresponding positive constants
$\delta_C=\delta_C(\varepsilon')<1$, $\Lambda_C=\Lambda_C(\varepsilon')<1$ and
$R_C=R_C(\varepsilon')>1$. We then put $\varepsilon'':= r_0\Lambda_C\delta_C
R_C^{-1}\sqrt{\varepsilon'\slash (m-1)}$ in \cite[Proposition 5.4]{NaberZhang}
to obtain positive constants $\delta_{NZ}(\varepsilon'')<1$ and
$r':=r_{NZ}(\varepsilon'')\in (\delta_{NZ}(\varepsilon''),1)$. Finally we put
$\delta_{AE}(\varepsilon):=\Psi_{NZ}(\delta_{NZ}|m)<<
\delta_{NZ}(\varepsilon'')$ --- here $\Psi_{NZ}>0$ is the uniform constant
obtained in \cite[Lemma 5.3]{NaberZhang} (see also Remark~\ref{rmk:
nonlocal_rank}) --- we point out that $\Psi_{NZ}(\delta|m)$ is monotone
increasing in $\delta$, as readily checked from the proof of \cite[Lemma
5.3]{NaberZhang}.

Now suppose $B_g(p,10)\subset M$ satisfies the assumptions (1) and (2) with
$\delta<\delta_{AE}$. For any $x\in B_g(p,7)$, since $B_g(x,2)\subset
B_g(x,10)$, by $d_{GH}\left(B_g(p,10),\mathbb{B}^k(10)\right)<\delta$ we have
\begin{align}\label{eqn: B2_GH}
d_{GH}\left(B_g(x,2),\mathbb{B}^k(2) \right)\ <\ \delta.
\end{align}

On the other hand, the given normal covering $\pi:X\to B_g(p,10)$ restricts to
a normal covering $\pi^{-1}(B_g(x,2))\to B_g(x,2)$ with deck transformation
group $G$. Therefore, with some $\tilde{x}\in \pi^{-1}(x)\subset X$ fixed, we
have
\begin{align*}
\widehat{G}_{\delta_{NZ}}(x)\ :=\ &\left\langle \gamma\in G:\
d_{\pi^{\ast}g}(\gamma.\tilde{x},\tilde{x})<2\delta_{NZ}\right\rangle.
\end{align*}
By \cite[Lemma 5.3 (ii)]{NaberZhang}, especially (\ref{eqn: nonlocal_rank}), 
we have $\rank\ \widehat{G}_{\delta_{NZ}}(x)\ge \rank\
\widehat{G}_{\Psi_{NZ}(\delta_{NZ}|m)}(p)$. Moreover, since 
$\delta<\delta_{AE}=\Psi_{NZ}(\delta_{NZ}|m)$, we have
$\widehat{G}_{\delta}(p)\le \widehat{G}_{\delta_{AE}}(p)$, whence the lower
bound
\begin{align}\label{eqn: B2_rank}
\rank\ \widehat{G}_{\delta_{NZ}}(x)\ \ge\ \rank\ \widehat{G}_{\delta}(p)\ =\
m-k.
\end{align} 
Recalling that $\delta_{NZ}=\delta_{NZ}(\varepsilon'')$, we could apply
\cite[Proposition 5.4]{NaberZhang} with (\ref{eqn: B2_GH}) and (\ref{eqn:
B2_rank}) to see 
\begin{align*}
d_{GH}\left(B_{\pi^{\ast}g}(\tilde{x},r'),\mathbb{B}^m(r')\right)\ \le\
\varepsilon''r'.
\end{align*}
Now consider the rescaled metric $\bar{g}:=\lambda^{-2}\pi^{\ast}g$ with 
\begin{align}\label{eqn: choosing_lambda}
 \lambda(\varepsilon)\ :=\ \min\left\{r_0, \Lambda_C, 
 \sqrt{\varepsilon'\slash (m-1)}, r'R_C^{-1}\right\},
\end{align}
we have for any $\tilde{x}\in \pi^{-1}(B_g(p,7))$ that 
\begin{align*}
d_{GH}\left(B_{\bar{g}}(\tilde{x},R_C),\mathbb{B}^m(R_C)\right)\ <\ \delta_C,
\end{align*}
and we have the Ricci curvature lower bound 
\begin{align}\label{eqn: gbar_Ricci}
\Rc_{\bar{g}}\ \ge\ -\min\left\{(m-1)\Lambda_C^2,\varepsilon'\right\}\bar{g}.
\end{align} 
Consequently, applying \cite[Main Lemma 2.1]{Colding97} we have for any
$\tilde{x}\in \pi^{-1}(B_g(p,7))$,
\begin{align}\label{eqn: gbar_vol}
\left|B_{\bar{g}}(\tilde{x},1)\right|_{\bar{g}}\ \ge\ (1-\varepsilon')
\omega_m.
\end{align}
By the volume ratio comparison (\ref{eqn: choosing_r_0}) we have
\begin{align}\label{eqn: gbar_vr}
\forall r\in (0,1],\ \forall \tilde{x}\in \pi^{-1}(B_g(p,7)),\quad
\left|B_{\bar{g}}(\tilde{x},r)\right|_{\bar{g}}\ \ge\
(1-\varepsilon)\omega_mr^m.
\end{align}
On the other hand, by (\ref{eqn: gbar_Ricci}), (\ref{eqn: gbar_vol}) and the
choice of $\varepsilon'$, we could apply \cite[Theorem 1.1]{CM18}
to see 
\begin{align}\label{eqn: gbar_isoperimetric}
\forall r\in (0,\varepsilon'],\quad
I_{\bar{g}}\left(B_{\bar{g}}(\tilde{x},r)\right)\ \ge\ (1-\varepsilon)I_m.
\end{align}
for any $\tilde{x}\in \pi^{-1}(B_g(p,7))$. Notice the scaling invariance of
these estimates.

Now we scale back to the original metric $\pi^{\ast}g$ and the estimates
(\ref{eqn: gbar_vr}) and (\ref{eqn: gbar_isoperimetric}) remain valid for
geodesic balls centered anywhere in $X$, with radii not exceeding
$r_{AE}:=\varepsilon'\lambda(\varepsilon)$.
By (\ref{eqn: choosing_lambda}) and the bound of $r'\in (\delta_{AE}',1)$ in
\cite[Proposition 5.8]{NaberZhang}, we have the following bound on $r_{AE}$,
solely determined by $m$ and $\varepsilon$:
\begin{align}
\min\left\{r_0(\varepsilon'),
\Lambda_C(\varepsilon'),\sqrt{\varepsilon'\slash (m-1)},
\delta'_{AE}(\varepsilon'')R_C(\varepsilon')^{-1}\right\}\varepsilon'\ \le\
r_{AE}\ \le\ \varepsilon',
\end{align}
with $\varepsilon'=\frac{\varepsilon}{16}\min\left\{1,C_{0,m}^{-1}\right\}$ and
$\varepsilon''$ determined by $\varepsilon'$ via Colding's theorem.
\end{proof}

 With the almost Euclidean isoperimetric constant estimate on the local
 universal covering space, we could apply the pseudo-locality theorem on the
 covering space and write down the details of proving Theorem~\ref{thm: main1}. 
\begin{proof}[Proof of Theorem~\ref{thm: main1}] 
The short-time existence of the Ricci flow is already shown in Lemma~\ref{lem:
start_RF}, and here we only need to bound the existence time from below by a
constant only depending on the dimension $m$. We fix an $\alpha\in (0,10^{-1})$
to begin our discussion.

Recall that the conformal factor $f_{\theta}$ we used in Lemma~\ref{lem:
start_RF} involves an extra parameter $\theta\in (0,\frac{1}{2})$, which we now
fix. Let $\epsilon':=\frac{\delta_{P}}{4}$, where
$\delta_{P}=\delta_P(m,\alpha)>0$ is the dimensional constant provided by
Theorem~\ref{thm: CTY11}. We then define (assuming $\delta_P\le 1$)
\begin{align}\label{eqn: define_theta}
\theta\ :=\
\frac{1}{2}\left(\frac{4-\delta_{P}}{4-2\delta_{P}}\right)^{\frac{1}{m(m-1)}}
-\frac{1}{2}
\end{align}
This dimensional constant is thus defined so that
$(1-2\theta)^{-m}(4-\delta_{P})= (4+\delta_{P})$. Once $\theta$ is chosen, we
will drop this subscript in our writing of $f=f_{\theta}$ and $w=w_{\theta}$,
given in Lemma~\ref{lem: start_RF}. 

We set $\delta_E(m,R,\alpha):=10^{-2}\delta_{AE}(\epsilon')R$.

Letting $\pi:X\to B_g(K,R)$ denote the universal covering map, we now pull the
metric $g$ back to $X$. Clarly $\Rc_{\pi^{\ast}g}\ge -(m-1)\pi^{\ast}g$ on $X$. 
We notice that the conformal change of the pull-back metric
is exactly the pull-back of the conformal change:
\begin{align*} 
\pi^{\ast}h\ =\
\pi^{\ast}(e^{2f}g)\ =\ e^{2f\circ \pi}\pi^{\ast}g,
\end{align*} 
and the conformal factor is $e^{f\circ \pi}=:e^{\tilde{f}}$. We also denote
$\tilde{\rho}:=\pi^{\ast}\rho$. Equipping $X_0:=\pi^{-1}(B_g(K,\frac{3R}{4}))$
with the covering metric $\pi^{\ast}h$, we have made it a complete Riemannian
manifold. Moreover, the bounds on the sectional curvature of $\pi^{\ast}h$
remain the same as that of $h$: by (\ref{eqn: kappa_C2}) we have
\begin{align}\label{eqn: sectional_pi_h}
\sup_{\wedge^2TX_0}|\mathbf{K}_{\pi^{\ast}h}|\ \le\
10^9k^2\left(C_0(m,R)^2R^{-2}+\kappa(K,g,R)+
\|\psi\|_{C^2(M)}\right) \theta^{-4}\ <\ \infty.
\end{align}
Therefore, we could start Ricci flow $(X_0,\tilde{h}(t))$ from the initial data 
$(X_0,\pi^{\ast}h)$ by Theorem~\ref{thm: Shi}. Notice that the Ricci flow
$\tilde{h}(t)$ is invariant under the action of $\pi_1(B_g(K,R))$, as
does the initial data $h$; therefore $(X_0,\tilde{h}(t))$ covers
$(B_g(K,\frac{3R}{4}),h(t))$, the original flow obtainted in Lemma~\ref{lem:
start_RF}, and consequently, the existence time of $h(t)$ is the same as the
existence time of $\tilde{h}(t)$, which we will bound from below via the
pseudo-locality theorem (Theorem~\ref{thm: CTY11}).

To apply Theorem~\ref{thm: CTY11} to the Ricci flow $(X_0,\tilde{h}(t))$,
we start with checking the relevant properties satisfied by
$(B_g(K,\frac{3R}{4}),g)$. By our assumption that
$\delta_E=10^{-2}\delta_{AE}R$, we have
for any $\tilde{x}\in X_0$,
\begin{align*}
d_{GH}\left(B_{\pi^{\ast}g}(\pi(\tilde{x}),10^{-1}R),
\mathbb{B}^k(10^{-1}R)\right)\ <\ 10^{-2}\delta_{AE}R,
\end{align*} 
and that $\rank\ \widehat{G}_{\delta_E}(\pi(\tilde{x}))=\rank\
\tilde{\Gamma}_{\delta_E}(\pi(\tilde{x}))=m-k$; therefore, after 
rescaling $\pi^{\ast}g \mapsto 10^{4}R^{-2}\pi^{\ast}g$ and applying
Lemma~\ref{lem: almost_Euclidean}, we see that
\begin{align}\label{eqn: almost_isoperimetric_2}
\forall \tilde{x}\in X_0,\ \forall r\in (0,10^{-2}r_{AE}R],\quad
I_{\pi^{\ast}g}\left(B_{\pi^{\ast}g}(\tilde{x},r)\right)\ \ge\
(1-\varepsilon')I_m.
\end{align}

Now we consider the corresponding bounds on the metric $\pi^{\ast}h$. We first
see that the scalar curvature is uniformly bounded from below. By the standard
formula, we have
\begin{align}\label{eqn: Sc_pi_h}
\begin{split}
\Sc_{\pi^{\ast}h}\ =\
&e^{-2\tilde{f}}\left(\Sc_{\pi^{\ast}g}
-\frac{4(m-1)}{m-2}e^{-\frac{m-2}{2}\tilde{f}} \Delta_{\pi^{\ast}g}
e^{\frac{m-2}{2}\tilde{f}}\right)\\
=\ &e^{-2\tilde{f}} \left[\Sc_{\pi^{\ast}g}-(m-1)
\left(\left(2w''+(m-2)(w')^2\right)|\nabla_{\pi^{\ast}
g}\tilde{\rho}|_{\pi^{\ast}g}^2+2w'\Delta_{\pi^{\ast}g}\tilde{\rho}
\right)\right].
\end{split}
\end{align} 
Since $\pi$ is a covering map, by (\ref{eqn: rho_C2}) we have 
\begin{align*}
R\|\nabla_{\pi^{\ast}g}\tilde{\rho}\|_{C^0(X)}
+R^2\|\Delta_{\pi^{\ast}g}\tilde{\rho}\|_{C^0(X)}\ \le\ 100C_0(m,R).
\end{align*}
Consequently, as $\Rc_{\pi^{\ast}g}\ge -(m-1)\pi^{\ast}g$, we have the scalar
curvature lower bound for $\pi^{\ast}h$ as
\begin{align}\label{eqn: Sc_pi_h_lb}
\Sc_{\pi^{\ast}h}\ \ge\ -10^2m(m-1)C_0(m,R)^2R^{-2}\theta^{-4}\ =:\
C_1(m,R,\theta).
\end{align}
Notice that here $\theta$ is already determined by $\alpha\in (0,10^{-1})$.

 Moreover, we need to control the isoperimetric constant of $\pi^{\ast}h$
 around any given point in $X_0$, and we only need to focus on the region
 $U:=\tilde{\rho}^{-1}((1-\frac{3}{2}\theta,1))\subset X_0$, since 
 $\pi^{\ast}g(\tilde{x})=\pi^{\ast}h(\tilde{x})$ whenever
 $\tilde{\rho}(\tilde{x})\le 1-\frac{3}{2}\theta$.
 Fixing any $\tilde{x}\in U$, let us denote  $s=\tilde{\rho}(\tilde{x})$ and
 define $U_s:=\tilde{\rho}^{-1}(s-\zeta,s+\zeta)$ for the moment, with
 $\zeta=\frac{1}{4}\theta(1-s)$. The key feature for the points in $U_{s}$ is
 the following metric equivalence:
  \begin{align}\label{eqn: metric_equivalence}
  \forall \tilde{y}\in U_{s},\quad 
e^{2w(s-\zeta)}\pi^{\ast}g(\tilde{y})\ \le\ \pi^{\ast}h(\tilde{y})\
\le\ e^{2w(s+\zeta)}\pi^{\ast}g(\tilde{y}).
\end{align}
  As we have $1-\frac{3}{2}\theta<s<1$, and as indicated in (\ref{eqn:
 seeking_beta}), we could consider all radii $r>0$ bounded as
 \begin{align}\label{eqn: good_scale}
 r\ <\ \min\left\{\frac{\theta^2R}{10^6C_0(m,R)},\
 10^{-2}r_{AE}(\varepsilon')R\right\}\ \le\
 \min\left\{\frac{\zeta e^{w(s-\zeta)}R}{10^4C_0(m,R)} ,\
 10^{-2} r_0(\varepsilon')R\right\},
 \end{align}
 where $\zeta=\frac{1}{4}\theta(1-s)$ and $r_0(\varepsilon')$ is
 defined as in (\ref{eqn: choosing_r_0}). 
 
 We now claim that $B_{\pi^{\ast}h}(\tilde{x},r)\subset U_{s}$: suppose
 otherwise, there is some $\tilde{y}\in B_{\pi^{\ast}h}(\tilde{x},r)\backslash
 U_{s}$, we could then let $\gamma:[0,1]\to X_0$ be a minimal 
 $\pi^{\ast}h$-geodesic connecting $\tilde{x}=\gamma(0)$ to
 $\tilde{y}=\gamma(1)$; since $\tilde{x}\in U_s$, there must be a first
 time $t_0\in (0,1)$ such that $\gamma(t_0)\in \partial U_s$ and
 $\gamma([0,t_0))\subset U_{s}$; by the continuity of $\rho$, we know that
 $|s-\tilde{\rho}(\gamma(t_0))|=\zeta$; but by (\ref{eqn: seeking_beta})
 and (\ref{eqn: metric_equivalence}) we have
 \begin{align}\label{eqn: ball_in_Us}
 \begin{split}
 \left|\tilde{\rho}(\tilde{x})-\tilde{\rho}(\gamma(t_0))\right|\ \le\
 &\int_0^{t_0}  \left|\pi^{\ast}g(\nabla_{\pi^{\ast}g}
 \tilde{\rho}(\gamma(t)), \dot{\gamma}(t))\right|\ \text{d}t  \\
 \le\  &\|u'\|_{C^0([0,1])} \|\nabla_g\psi\|_{C^0(B_g(K,R))}
 d_{\pi^{\ast}g}(\tilde{x},\gamma(t_0))\\
  \le\  &10^3C_0(m,R)R^{-1}d_{\pi^{\ast}g}(\tilde{x},\gamma(t_0))\\
 \le\ &10^3C_0(m,R)R^{-1}e^{-w(s-\zeta)}r\\
 \le\ &\frac{\zeta}{10};
 \end{split}
 \end{align}
 this provides a contradiction, and the claim is proven. 
 
 Consequently, on $B_{\pi^{\ast}h}(\tilde{x},r)$ we also have the uniform metric
 equivalence (\ref{eqn: metric_equivalence}), and so we could easily compare the
 volume of subsets in the ball, measured in the two metrics. Now for any region 
 $\Omega\subset B_{\pi^{\ast}h}(\tilde{x},r)$, we could estimate
 \begin{align*}
 |\Omega|_{\pi^{\ast}h}\ \le\
 e^{mw(s+\zeta)}|\Omega|_{\pi^{\ast}g}\quad 
 \text{and}\quad |\partial \Omega|_{\pi^{\ast}h}\ \ge\
 e^{(m-1)w(s-\zeta)}|\partial \Omega|_{\pi^{\ast}g},
 \end{align*}
 and by (\ref{eqn: almost_isoperimetric_2}), we get the \emph{scaling-invariant}
 estimate
 \begin{align*}
 |\partial \Omega|_{\pi^{\ast}h}^m\ \ge\
 \frac{(1-\varepsilon')I_m}{(1+2\theta)^{m(m-1)}}|\Omega|_{\pi^{\ast}h}^{m-1}.
 \end{align*}
 Since $\Omega\subset B_{\pi^{\ast}h}(\tilde{x},r)$ is arbitrarily chosen, by
 our choice of the constants $\varepsilon'$ and $\theta$, we get the control of
 the isoperimetric constant:
 \begin{align}\label{eqn: h_isoperimetric}
 I_{\pi^{\ast}h}\left(B_{\pi^{\ast}h}(\tilde{x},r)\right)\ \ge\ (1-\delta_P)I_m 
 \end{align}
 for any $\tilde{x}\in X_0$ and any $r>0$ in the range specified by (\ref{eqn:
 good_scale}). Notice that the estimate (\ref{eqn: h_isoperimetric}) remains
 unchanged under rescaling of the metric.

Now by (\ref{eqn: Sc_pi_h_lb}), (\ref{eqn: almost_isoperimetric_2}) for the
region where $\pi^{\ast}h=\pi^{\ast}g$ and (\ref{eqn: h_isoperimetric})
for the region $U$, we could apply the pseudo-locality thoerem to obtain a
lower bound of the existence time: Consider the dimensional constant
\begin{align*}
\mu(m,R,\alpha)\ :=\ \min\left\{C_1(m,R,\theta),\ 
\frac{\theta^2R}{10^6C_0(m,R)},\
10^{-2}r_{AE}R\right\},
 \end{align*} 
 and rescale the the flow $(X_0,\tilde{h}(t))$ to $(X_0,\bar{h}(s))$,
 with $s=\mu^{-2}t$ and $\bar{h}(s):=\mu^{-2}\tilde{h}(s)$. Now we have 
\begin{align}
 \Sc_{\bar{h}(0)}\ge -1\quad \text{and}\quad
\forall \tilde{x}\in X_0,\ 
I_{\bar{h}(0)}\left(B_{\bar{h}(0)}(\tilde{x},1)\right)\ \ge\ (1-\delta_{P})I_m.
\end{align}
Applying Theorem~\ref{thm: CTY11}, we see that the the existence time of
the rescaled flow $(X_0,\bar{h}(s))$ is bounded below by $\varepsilon_{P}^2>0$,
as previously discussed. Now scaling back, we see that the
existence time $T$ for the flow $(X_0,\tilde{h}(t))$ is bounded below as 
\begin{align}
T\ \ge\ \mu^2\varepsilon_{P}^{2}\ =:\ \varepsilon_{E}^2(m,R,\alpha),
\end{align}
which is a constant only depending on $m$ and $\alpha$, whence the desired
lower bound of the existence time for the original Ricci flow. Moreover, we
have the curvature estimate
\begin{align*}
\forall s\in (0,\varepsilon_P^2],\quad
\sup_{X_0}\left|\Rm_{\bar{h}(s)}\right|_{\bar{h}(s)}\ \le\ \alpha
s^{-1}+\varepsilon_P^{-2}.
\end{align*}
Rescaling back and restricting our attention to $B_g(K,\frac{R}{4})$, which is
unaffected by the conformal transformation, we get the desired curvature control.
\end{proof}

\begin{remark}
The application of the pseudo-locality theorem in proving the existence of Ricci
flows seems to appeare in the work \cite{Topping10} of Topping for the first
time. Here we notice that checking the completeness of the pull-back metric on
the local universal covering is necessary for the application of
Theorem~\ref{thm: CTY11}. An example where the pseudo-locality theorem
fails for incomplete Ricci flows is given by Topping; see e.g. \cite[Example
0.3]{HKRX18}.
\end{remark}

\begin{remark}\label{rmk: TianWang}
It seems that the initially assumed Ricci lower bound allows us to apply the 
version of the pseudo-locality theorem due to Tian and the second-named author 
(\cite[Proposition 3.1]{TianWang}) directly, as is done in our previous result
\cite[Theorem 1.4]{HW20a}. But the conformal transformation involved here
prevents us from doing so --- we do not have a uniform $C^{2}$ control of
the cut-off function $\psi$ in the definition of the conformal
factor: with $\tilde{f}=\pi^{\ast}(w\circ u\circ \psi)$, the lower
bound of the conformally transformed Ricci curvature
\begin{align}\label{eqn: Rc_lb_h}
\Rc_{\pi^{\ast}h}\ =\
\Rc_{\pi^{\ast}g}-(m-2)\left(\nabla^2\tilde{f}-\nabla \tilde{f}\otimes
\nabla \tilde{f}\right) -\left(\Delta \tilde{f} +(m-2)|\nabla
\tilde{f}|^2\right)\pi^{\ast}g
\end{align}
depends on the full Hessian bound $\|\nabla^2_g\psi\|_{C^0(B_g(K,R)}$, which
has no uniform \emph{a priori} estimate. In contrast, the scalar curvature of
$\pi^{\ast}h$ only depends on $\|\Delta_g\psi\|_{C^0(B_g(K,R))}$, which is
indeed uniformly bounded and leaves a chance to apply Perelman's original
version of the pseudo-locality theorem.
\end{remark}

Extracting the technical essence involved in the proof of Theorem~\ref{thm:
main1}, we can generalize this theoerm to the setting where only the initial
\emph{scalar curvature} lower bound is known.
\begin{theorem}\label{thm: sc}
For any $\alpha\in (0,10^{-1})$, $C>0$ and $R>0$ there are constants 
$\delta_{SC},\varepsilon_{SC}\in (0,1)$ solely determined by $m, C, R$ and
$\alpha$, such that if $K$ is a compact subset of an $m$-dimensional Riemannian
manifold $(M,g)$ with scalar curvature bounded below by $-1$, satisfying
\begin{enumerate}
  \item the universal covering space of $B_g(K,R)$ has isoperimetric
  constant no less than $(1-\delta_{SC})I_m$ everywhere at scale $1$, and
  \item there is a cut-off fuction $\rho$ supported in $B_g(K,\frac{3R}{4})$
  such that  
  \begin{align*}
  0\ \le \rho\ \le\ 1,\quad \rho|_{B_g(K,\frac{R}{4})}\equiv
  1\quad \text{and}\quad R\|\rho\|_{C^1(M)}+R^2\max_{M}\Delta_g\rho\ \le\ C,
  \end{align*}  
\end{enumerate}
then there is a Ricci flow solution with initial data $(B_g(K,\frac{R}{4}),g)$,
existing at least up to $\varepsilon_{SC}^2$, and satisfying
\begin{align*}
\forall t\in (0,\varepsilon_{SC}^2],\quad
\sup_{B_g(K,\frac{R}{4})}\left|\Rm_{g(t)}\right|_{g(t)}\ \le\ \alpha
t^{-1}+\varepsilon_{SC}^{-2}.
\end{align*}
\end{theorem}
\begin{proof}[Sketch of proof]
To indicate the proof, while Lemma~\ref{lem: start_RF} directly enables us to
start a Ricci flow with initial data $(K,g)$, we notice that the conformally
transformed metric using $\tilde{\rho}$ on the universal covering of
$B_g(K,R)$ has scalar curvature lower bound given by (\ref{eqn: Sc_pi_h}), and
Assumption (2) provides a uniform lower bound of this scalar curvature, just as
(\ref{eqn: Sc_pi_h_lb}); on the other hand, Assumption (1) together with the
$\|\rho\|_{C^1(M)}$ bound ensure that the isoperimetric constant of the
conformally transformed metric (on the universal covering) is sufficiently
close to the Euclidean isoperimetric constant, verifying (\ref{eqn:
h_isoperimetric}) --- we have all the ingredients ready to apply the
pseudo-locality theorem (Theorem~\ref{thm: CTY11}) and obtain the desired 
existence time and curvature bounds of the Ricci flow.
\end{proof}
Similar to Theorem~\ref{thm: main1}, this theorem may serve as a smoothing tool
in the study of the uniform behavior of Riemannian manifolds whose scalar
curvature are locally bounded below, a program initiated by
Gromov~\cite{Gromov18}; see also \cite{Bamler16, Sormani17}.

\section{Initial data locally collapsing to orbifold model spaces}
This section is devoted to the proof of Theorem~\ref{thm: main2}, based on the
discussion in the last section. We will prove the theorem by a contradiction
argument. Now suppose for some $R>0$ fixed we have a sequence of data
$\left\{K_i\subset (M_i,g_i)\right\}$ satisfying the assumptions of the
theorem, but with the Ricci flow existence time on $B_{g_i}(K_i,\frac{R}{4})$
decaying to zero. Notice that by Lemma~\ref{lem: start_RF}, as long as
the ambient manifold has Ricci curvature uniformly bounded below, we could
start the Ricci flow regardless of the specific geometry of the region
$B_{g_i}(K_i,\frac{R}{4})$. Now letting $\pi_i: X_i\to B_{g_i}(K_i,R)$ denote
the universal covering equipped with the covering metric $\pi_i^{\ast}g_i$ and
setting $X_{i,0}=\pi_i^{-1}(B_{g_i} (K_i,\frac{3R}{4}))$, if $X_{i,0}$ were 
everywhere almost Euclidean at a fixed scale, i.e. (\ref{eqn:
almost_isoperimetric_2}) holds in the setting of Theorem~\ref{thm: main2},
then applying Theorem~\ref{thm: CTY11} to the conformally transformed
metric, we could uniformly bound the Ricci flow existence time from below. 
Therefore, by the contradiction hypothesis, we could find $p_i\in
B_{g_i}(K_i,\frac{3R}{4})$ such that geodesic balls centered at $\tilde{p}_i\in
\pi_i^{-1}(p_i)\subset X_{i,0}$ with a fixed size cannot be locally almost
Euclidean.

More specifically, we may assume that there is some $G<O(k)$ with $|G|\le l$
such that 
\begin{align*}
d_{GH}\left(B_{g_i}(p_i,4^{-1}R), \mathbb{B}^k(4^{-1}R)\slash
G \right)\ =\ \delta_i\ \to\ 0\quad \text{as}\quad i\ \to\ \infty,
\end{align*}
however, for $\varepsilon'=\frac{\delta_P(\alpha)}{4}$ and any $r\in
(0,\frac{R}{10})$ fixed, the isoperimetric constant satisfies 
\begin{align}\label{eqn: contradiction_epsilon}
I_{\pi^{\ast}g_i}\left(B_{\pi^{\ast}g_i}(\tilde{p}_i,r)\right)\
\le\ (1-\varepsilon')I_m.
\end{align}
 On the other hand, from Assumption (1) of Theorem~\ref{thm: main2}, we
have surjective group homomorphisms $\phi_i: \pi_1(B_{g_i}(K_i,R),p_i)
\twoheadrightarrow G$, and consequently we have 
$\pi_1(B_{g_i}(K_i,R),p_i)\slash \ker \phi_i=H_i\cong G$. To facilitate
our argument, we will let $H$ denote all $H_i$ as they are all isomorphic to
one another, and let $\phi:H\xrightarrow{\approx} G$ denote the group
isomorphism induced by $\phi_i$ for $i$ large enough (after possibly passing to
a sub-sequence); we will also perform the rescaling $g\mapsto 10^5R^{-2}g=:g'$.

Since $\ker \phi_i \trianglelefteq \pi_1(B_{g_i}(K_i,R),p_i)$, the
covering map $\pi_i$, when restricted to the $g'$-metric closure 
$\overline{\pi_i^{-1}(B_{g_i'}(p_i,50))}$, induces a normal covering
$\pi_{i,0}:\overline{\pi_i^{-1}(B_{g_i'}(p_i,50))}\to
Y_i:=\overline{\pi_i^{-1}(B_{g_i'}(p_i,50))}\slash \ker \phi$, equipped with
the quotient metric $\bar{g}_i$ of $\pi^{\ast}_ig_i'$. Clearly, the deck
transformation group of this covering is $\ker \phi_i$. Moreover, the
deck transformation group of the covering $\bar{\pi}_i: Y_i\to
B_{g_i'}(p_i,50)\equiv Y_i\slash H$ is nothing but $H$, which is isomorphic to
the local orbifold group $G$. Clearly,  $H\le Isom(Y_i,\bar{g}_i)$.

To produce a contradiction, we will show that for $i$ large enough, 
$\left\{B_{\bar{g}_i}(\bar{p}_i,10)\subset Y_i\right\}$ is sufficiently
Gromov-Hausdorff close to $\mathbb{B}^k(10)$, for some $\bar{p}_i\in
\bar{\pi}_i^{-1}(p_i)\subset Y_i$ fixed. Applying Lemma~\ref{lem:
almost_Euclidean} to the normal covering $\pi_{i,0}$, we could then find an
estimate contradicting (\ref{eqn: contradiction_epsilon}) for some $r>0$ fixed.
To prove such Gromov-Hausdorff proximity we rely on yet another
contradiction argument: assuming
\begin{align}\label{eqn: contradiction_delta_AE}
\liminf_{i\to \infty} d_{GH}\left(B_{\bar{g}_i}(\bar{p}_i,10), \mathbb{B}^k(10)
\right)\ \ge\ \delta_{AE}\left(2^{-1}\varepsilon'\right),
 \end{align} 
  we will deduce a contradiction from the situation summarized in the
  following diagram:
\begin{align}\label{eqn: orbifold_diagram}
\begin{split}
\begin{xy}
<0em,0em>*+{(Y_i,\bar{g}_i)}="x", 
<-11em, 0em>*+{(\pi_i^{-1}(B_{g_i'}(p_i,50)),\pi_i^{\ast}g_i')}="w",
<0em,-5em>*+{(B_{g_i'}(p_i,50),g_i')}="v", 
<11em,0em>*+{(\overline{\mathbb{B}^k(50)},g_{Euc})}="y",
<11em,-5em>*+{(Z,d_Z)}="u",
 "w";"x" **@{-} ?>*@{>}?<>(.5)*!/_0.5em/{\scriptstyle \pi_{i,0}}, 
 "w";"v" **@{-} ?>*@{>} ?<>(.5)*!/_0.5em/{\scriptstyle \pi_i}, 
 "y";"u" **@{-} ?>*@{>} ?<>(.5)*!/_0.5em/{\scriptstyle \slash G}, 
 "v";"u" **@{-} ?>*@{>} ?<>(.5)*!/_0.5em/{\scriptstyle \delta_i-GH\ close},
  "u";"v" **@{-} ?>*@{>} ?<>(.5)*!/_0.5em/{ },
 "x";"y" **@{-} ?>*@{>} ?<>(.5)*!/_0.5em/{\scriptstyle \delta_{AE}-GH\
 apart}, "y";"x" **@{-} ?>*@{>} ?<>(.5)*!/_0.5em/{ }, 
  "x";"v" **@{-} ?>*@{>} ?<>(.5)*!/_0.5em/{\scriptstyle \bar{\pi}_i}, 
\end{xy}
\end{split}
\end{align}
 
Since $\Rc_{\bar{g}_i}\ge -(m-1)\bar{g}_i$, we have
$\left\{(Y_i,\bar{g}_i,\bar{p}_i)\right\}$ sub-converges to some $(Y_{\infty},
d_{\infty},\bar{p}_{\infty})$ in the pointed and $H$ equivariant
Gromov-Hausdorff topology. Especially, since $B_{g'_i}(p_i,50)\to
Z:=\overline{\mathbb{B}^k(50)}\slash G$ in the pointed Gromov-Hausdorff
topology, we see that $Y_{\infty}\slash H\equiv Z$. We will identify
$Y_{\infty}$ with its completion under $d_{\infty}$, making it a compact and
connected metric space.

We now aim at proving the following
\begin{lemma}\label{lem: isometry}
$B_{d_{\infty}}(\bar{p}_{\infty},10)\ \equiv\ \mathbb{B}^k(10)$.
\end{lemma} 
It seems not obvious how to re-construction
$B_{d_{\infty}}(\bar{p}_{\infty},10)$ directly out of the quotient and the group
action, and we need to rely on the metric measure property of $Y_{\infty}$, 
appealing to the theory of $RCD$ spaces. Eventually, we will show that
$\overline{B_{d_{\infty}}(\bar{p}_{\infty},20)}$ is a non-collapsing $RCD(0,k)$
space, but as its $k$-dimensional Hausdorff measure is equal to the volume of
$\mathbb{B}^k(20)$, by the recent work \cite{DPG18} we can conclude that
$B_{d_{\infty}}(\bar{p}_{\infty},10)$ is isometric to the $k$-Euclidean
$10$-ball.

\begin{proof}
Notice that if $G<O(k)$ fixes no point on $\mathbb{S}^{k-1}$, then
$\mathbb{S}^{k-1}\slash G$ is a manifold and $Y_{\infty}\equiv
\overline{\mathbb{B}^k(50)}$. The situation is more complicated when $G$ does
not act on $\mathbb{S}^{k-1}$ freely; see Remark~\ref{rmk: general_orbifold}.
In this case, since $G<O(k)$, the fixed point set of any element of $G$ is a
vector sub-space of $\mathbb{R}^k$.
Especially, the singular part $\Sigma$ of $Z$ is cut out by sub-spaces of
$\mathbb{R}^k$. Consequently, every $x\in Z$ has a small radius
$r_x>0$ such that $B_{d_Z}(x,r_x)$ is the geodesic $r_x$-ball centered at
$x$ in a metric cone with vertex $x$. 


On the other hand, for the finite quotient $\bar{\pi}_{\infty}:Y_{\infty}\to Z$ 
we may decompose $Y_{\infty}$ into a union of the orbit of a fundamental domain
$\Lambda$, i.e. $Y=\cup_{\gamma\in H}\gamma.\Lambda$, and for any regular point 
$y\in Y_{\infty}$ there is a \emph{unique} $\gamma\in H$ such that $\gamma.y\in
\Lambda$ --- here we put $\mathcal{R}:=\left\{y\in Y_{\infty}:\ \forall
\gamma\in H,\ \gamma.y\not=y\right\}$ as the regular part of $Y_{\infty}$,
and the singular part is $\mathcal{S}:=Y_{\infty}\backslash
\mathcal{R}$. Notice that for $\gamma\in H\backslash \{Id_{Y_{\infty}}\}$, 
$\gamma.\Lambda$ may intersect $\Lambda$ non-trivially, and the intersection is
contained in $\mathcal{S}$. Moreover, $\bar{\pi}_{\infty}|_{\Lambda}$ is a
bijective local isometry onto $Z$, and we notice that $\bar{\pi}_{\infty}$ maps
$\mathcal{R}$ and $\mathcal{S}$ respectively to the regular and singular
parts of $Z$. It is also clear that $\mathcal{R}$ is an open subset.

We begin with the following
\begin{claim}\label{clm: fix_center}
$\bar{p}_{\infty}$ is a fixed point under the action of $H\le
Isom(Y_{\infty},\bar{p}_{\infty})$. 
\end{claim}
\begin{proof}[Proof of the claim]
We begin with considering any $\gamma\in H$ such that
$\overline{\gamma.\Lambda}\cap\overline{\Lambda}\not=\emptyset$ --- by the
connectedness of $Y_{\infty}$ we can always find such a $\gamma$. 
If $\gamma.\bar{p}_{\infty}\not=\bar{p}_{\infty}$, then
$\gamma.\bar{p}_{\infty}\in \partial \Lambda\backslash \{\bar{p}_{\infty}\}$, as
$int(\Lambda)$ consists of regular points. Since
$\bar{\pi}_{\infty}(\bar{p}_{\infty})
=\bar{\pi}_{\infty}(\gamma.\bar{p}_{\infty})=z\in \Sigma$, we must have
$\gamma.\bar{p}_{\infty}\in \mathcal{S}$. As the fixed point set, $\mathcal{S}$
must be totally geodesic, and thus the minimal geodesic segment $\sigma$
connecting $\bar{p}_{\infty}$ to $\gamma.\bar{p}_{\infty}$ has its interior
entirely lying in $\mathcal{S}$. Since $\bar{\pi}_{\infty}(\mathcal{S})=\Sigma$ 
and $\bar{\pi}_{\infty}(\bar{p}_{\infty})
=\bar{\pi}_{\infty}(\gamma.\bar{p}_{\infty})=z$, under the local isometry
$\bar{\pi}_{\infty}$ the geodesic segment $\sigma$ becomes a smooth geodesic
loop $\bar{\pi}_{\infty}\circ \sigma$ based at $z\in Z$, whose interior lying in
$\Sigma$. But this is impossible because any geodesic segment of $\Sigma$
emanating from the vertex is a straight line sigment, and in particular no
geodesic loop in $\Sigma$ can be based at $z\in Z$.

Setting $H_1=\left\{\gamma\in H:\ \overline{\gamma.\Lambda}\cap
\overline{\Lambda}\not=\emptyset\right\}$, the same arguement shows that if 
$\overline{\gamma.\Lambda}\cap \overline{\cup_{\gamma'\in
H_1}\gamma'.\Lambda}\not=\emptyset$, then
$\gamma.\bar{p}_{\infty}=\bar{p}_{\infty}$. For each $j\ge 1$, setting
$H_{j+1}=\left\{\gamma\in H:\ \exists \gamma'\in \cup_{i=1}^jH_i,\
\overline{\gamma.\Lambda}\cap \overline{\gamma'.\Lambda}
\not=\emptyset\right\}$, then we could inductively show that
$\gamma.\bar{p}_{\infty}=\bar{p}_{\infty}$ if $\gamma\in H_{j+1}$. Since
$Y=\cup_{\gamma\in H}\gamma.\Lambda$ and $Y$ is connected, we have
$\cup_{j=1}^{l'}H_j=H$ for some $l'\le l$, and thus we have shown that
$\gamma.\bar{p}_{\infty}=\bar{p}_{\infty}$ for any $\gamma\in H$. 
\end{proof}

From this claim, we see that $H$ acts on each geodesic ball centered at
$\bar{p}_{\infty}$, and $Y_{\infty}= \overline{B_{d_{\infty}}
(\bar{p}_{\infty},50)}$. Moreover, $H$ sends a minimal geodesic emanating from
$\bar{p}_{\infty}$ to another minimal geodesic emanating from
$\bar{p}_{\infty}$, and since $Z$ is a metric cone over $\mathbb{S}^{k-1}\slash
G$, $Y_{\infty}$ is also a geodesic ball in a metric cone centered at the vertex
$\bar{p}_{\infty}$. 

In the same vein, if $\bar{x}\in \mathcal{S}$, then its isotropy
group $H_{\bar{x}}:=\left\{\gamma\in H:\ \gamma.\bar{x}=\bar{x}\right\}$ acts on
a small ball around $\bar{x}$: since $|H|\le l$ and
$d_{\infty}(\gamma.\bar{x},\bar{x})>0$ for any $\gamma\not\in H_{\bar{x}}$, we
can find $\bar{r}_{\bar{x}}\in (0,r_{\bar{\pi}_{\infty}(\bar{x})})$ such that
$\gamma.B_{d_{\infty}}(\bar{x},\bar{r}_{\bar{x}})\cap
B_{d_{\infty}}(\bar{p}_{\infty},\bar{r}_{\bar{x}})=\emptyset$ whenever
$\gamma\in H\backslash H_{\bar{x}}$, but as $H_{\bar{x}}\le
Isom(Y_{\infty},d_{\infty})$, we know that $H_{\bar{x}}$ acts on
$B_{d_{\infty}}(\bar{x},\bar{r}_{\bar{x}})=:B_{\bar{x}}$ by isometries.
Moreover, the collection $\left\{\gamma.B_{\bar{x}}:\ \gamma\in H\right\}$ is in
bijective correspondence with the left coset of $H_{\bar{x}}$. Since $\gamma
H_{\bar{x}}\gamma^{-1}$ acts isometrically on $\gamma.B_{\bar{x}}$, and
$B_{\bar{x}}\slash H_{\bar{x}}\equiv (\gamma.B_{\bar{x}})\slash (\gamma
H_{\bar{x}}\gamma^{-1})$ for any $\gamma\in H$, we have $B_{\bar{x}}\slash
H_{\bar{x}}\equiv \bar{\pi}_{\infty}(B_{\bar{x}})\subset Z$; in fact we also
have $B_{\bar{x}}\slash H_{\bar{x}}\equiv B_{d_Z}(\bar{\pi}_{\infty} (\bar{x}),
\bar{r}_{\bar{x}})$ as $B_{\bar{x}}$ is a geodesic ball and
$\bar{\pi}_{\infty}$ is taking quotient by isometries. We further notice that
$B_{d_Z} (\bar{\pi}_{\infty}(\bar{x}),\bar{r}_{\bar{x}})$, as a geodesic ball
in a metric cone $Z$, is itself the geodesic $\bar{r}_{\bar{x}}$-ball in a
metric cone with vertex $\bar{\pi}_{\infty}(\bar{x})$, and consequently we know
that $B_{\bar{x}}$ is the geodesic $\bar{r}_{\bar{x}}$-ball centered at
$\bar{x}$ in a metric cone with vertex $\bar{x}\in Y_{\infty}$. Especially, if
$\bar{x}\in \mathcal{S}$ has a tangent cone isometric to $\mathbb{R}^k$, then
we must have $B_{\bar{x}}\equiv \mathbb{B}^k(\bar{r}_{\bar{x}})$. On the other
hand, it is clear that for any $\bar{x}\in \mathcal{R}$,
$B_{d_{\infty}}(\bar{x},\bar{r}_{\bar{x}})\equiv
\mathbb{B}^k(\bar{r}_{\bar{x}})$ with $\bar{r}_{\bar{x}}=
d_{\infty}(\bar{x},\mathcal{S})$. We will also let $B_{\bar{x}}$
denote $B_{d_{\infty}}\left(\bar{x}, d_{\infty}(\bar{x},\mathcal{S})\right)
\equiv \mathbb{B}^k(\bar{r}_{\bar{x}})$ when $\bar{x}\in \mathcal{R}$. 

On the other hand, notice that every point in the regular part $\mathcal{R}$ has
(any of) its tangent cone isometric to $\mathbb{R}^k$. According to the work of
Colding and Naber \cite{ColdingNaber}, the regular part of $Y_{\infty}$ is very
well connected. Letting $\mathcal{H}^k$ denote the $k$-dimensional Hausdorff
measure, we have the following 
\begin{claim}\label{clm: regular_connected}
For $\mathcal{H}^k\times \mathcal{H}^k$-a.e. pair of points $(x,y)\in
\mathcal{R}_{20}\times \mathcal{R}_{20}$, there is a minimal geodesic in
$B_{d_{\infty}}(\bar{p}_{\infty}, 40)$ with all of whose interior points having
the unique tangent cone $\mathbb{R}^k$. Here we employ the notation 
$\mathcal{R}_{r}:=\mathcal{R}\cap B_{d_{\infty}}(\bar{p}_{\infty},r)$ for any
$r\in (0,50)$.
\end{claim}
\begin{proof}[Proof of the claim]
Notice that the underlying manifolds $(Y_i,\bar{g}_i)$ are complete with
boundary, and by \cite[Theorem 1.1]{ColdingNaber} we know that as long as an
limiting minimal geodesic is contained in $B_{d_{\infty}}(\bar{p}_{\infty},
40)$, connecting two regular points whose tangent cones are $\mathbb{R}^k$,
and can be extended slightly towards both ends, then the tangent cones centered
at its interior are all isomtric to $\mathbb{R}^k$. Here we point out that
Colding and Naber's estimates are uniform (independent of $\{Y_i\}$) as long as
the limit geodesic in consideration is contained within $B_{d_{\infty}}
(\bar{p}_{\infty},40)$, since the Li-Yau gradient estimate
\cite[Theorem 1.2]{LiYau} and Harnack inequality \cite[Theorem 2.1]{LiYau} for
manifolds with boundary are uniform as long as we stay a definite distance away
from the boundary points. In fact, by \cite[Theorem 1.20]{ColdingNaber}, for
$\nu\times \nu$-a.e. pair of points $(x,y)\in \mathcal{R}_{20}\times 
\mathcal{R}_{20}$, since the minimal geodesic connecting them is entirely
contained in $B_{d_{\infty}}(\bar{p}_{\infty},40)$, every interior point of
this geodesic has its tangent cone isometric to $\mathbb{R}^k$. Here $\nu$ is
the renormalized measure, defined for any geodesic ball
$B_{d_{\infty}}(x,r)\subset Y_{\infty}$ by 
$\nu(B_{d_{\infty}}(x,r))=\lim_{i\to \infty}
\left|B_{\bar{g}_i}(\bar{p}_i,1)\right|_{\bar{g}_i}^{-1}
\left|B_{\bar{g}_i}(x_i,r)\right|_{\bar{g}_i}$,
where $B_{\bar{g}_i}(x_i,r)\xrightarrow{pGH}B_{d_{\infty}}(x,r)$ as $i\to
\infty$. By the volume comparison it is clear that $\nu(B_{d_{\infty}}(x,r))>0$
for any $x\in \mathcal{R}_{20}$ and $r\in (0,1)$, implying that for
$\mathcal{H}^k\times \mathcal{H}^k$-a.e. pair of points in
$\mathcal{R}_{20}\times \mathcal{R}_{20}$, the minimal geodesic connecting
these two points has tangent cone $\mathbb{R}^k$ at all of its interior points.
\end{proof}

Notice that if there is a minimal geodesic connecting two points in
$\mathcal{R}$ and passing through $\bar{p}_{\infty}$, then the tangent cone at
$\bar{p}_{\infty}$ is isometric to $\mathbb{R}^k$, but since
$B_{d_{\infty}}(\bar{p}_{\infty}, 20)$ is a geodesic ball in a metric cone
centered at the vertex $\bar{p}_{\infty}$, we must have
$B_{d_{\infty}}(\bar{p}_{\infty},20)\equiv \mathbb{B}^k(20)$, and the lemma is
proven. 

In general, there may be no minimal geodesic connecting regular
points and passing through $\bar{p}_{\infty}$, but the same reasoning shows 
that when a unit-speed minimal geodesic $\sigma:[0,r]\to 
B_{d_{\infty}}(\bar{p}_{\infty}, 40)$ connects two points in
$\mathcal{R}_{20}$, then it has a small tubular neighborhood locally
diffeomorphic to $(-\varepsilon,\varepsilon)\times
\mathbb{B}^{k-1}(\varepsilon)$ for some $\varepsilon>0$: since every interior
point of $\sigma$ is the center of a small geodesic ball isometric to one
centered at the cone vertex in a metric cone, and the regular-convexity
ensures that the tangent cone at each interior point is isometric to
$\mathbb{R}^k$, we see that each interior point $\sigma(t)$ has a small radius
$\bar{r}_{\sigma(t)}>0$ such that $B_{d_{\infty}} (\sigma(t),
\bar{r}_{\sigma(t)})\equiv \mathbb{B}^k(\bar{r}_{\sigma(t)})$; by the
compactness of $\sigma([0,r])$ we can find a minimal radius
$\bar{r}_{\sigma}>0$ such that the $\bar{r}_{\sigma}$-tubular neighborhood of
$\sigma([0,r])$ is locally diffeomorphic to the product mentioned above. In
fact, since $\sigma$ is a distance minimizer when restricted in each
$B_{d_{\infty}} (\sigma(t), \bar{r}_{\sigma(t)})$, it satisfies the Euclidean
geodesic equation, and thus becomes a straight line segment, i.e. we can
identify the $\bar{r}_{\sigma}$-tubular neighborhood of $\sigma([0,r])$
isometrically as
\begin{align}\label{eqn: good_nbhd}
B_{d_{\infty}}(\sigma([0,r]),\bar{r}_{\sigma})\ \equiv\ \left\{\vec{v}\in
\mathbb{R}^k:\ \exists t\in [0,r],\ \left|\vec{v}-t\vec{e}\right|
<\bar{r}_{\sigma}\right\},
\end{align} 
where $\vec{e}\in \mathbb{R}^k$ is a unit vector.

 Intuitively speaking, such tubular neighborhoods provide sufficiently regular
 ``tunnels'' that connect regular points in different fundamental domains of
 $B_{d_{\infty}}(\bar{p}_{\infty},20)$, and such tunnels exist in abundance as
 $\mathcal{R}$ is of full Hausdorff measure in each fundamental domain. The fact
 that $\mathcal{R}_{20}$ is very well connected enables us to prove the
 following segment inequality.
 \begin{claim}\label{clm: segment}
For any $u\in L^1(B_{d_{\infty}}(\bar{p}_{\infty},40))$ and any
$\bar{x}_0, \bar{y}_0\in \mathcal{R}$ with
$r=d_{\infty}(\bar{x}_0,\bar{y}_0)$, we have
\begin{align}\label{eqn: segment}
\int_{B_{d_{\infty}}(\bar{x}_0,r)} \int_{B_{d_{\infty}}(\bar{y}_0,r)}
\mathcal{F}_u(\bar{x},\bar{y})\
\text{d}\mathcal{H}^k(\bar{y})\text{d}\mathcal{H}^k(\bar{x})\ \le\ 2^{k+3}r\
\mathcal{H}^k\left(B_{d_{\infty}}(\bar{y}_0,4r)\right)
\int_{B_{d_{\infty}}(\bar{y}_0,4r)} u\ \text{d}\mathcal{H}^k,
\end{align}
where for any $\bar{x},\bar{y}\in \mathcal{R}$ we define 
$\mathcal{F}_u(\bar{x},\bar{y}):=\inf_{\sigma_{\bar{x}\bar{y}}}
\int_{\sigma_{\bar{x}\bar{y}}}u$, and the infimum is taken over all minimal
geodesics $\sigma_{\bar{x}\bar{y}}$ connecting $\bar{x}$ and $\bar{y}$ and
entirely contained in $\mathcal{R}_{40}$.
\end{claim}
 
 We need to reprove such an inequality, rather than directly applying the
 original one due to Cheeger and Colding \cite{ChCoIII}, because (\ref{eqn:
 segment}) is considered with respect to the $k$-dimensional Hausdorff measure,
 rather than the renormalized measure $\nu$. 
 \begin{proof}[Proof of the claim]
 We put
 $\mathcal{F}_u(\bar{x},\bar{y})=\mathcal{F}_u^+(\bar{x},\bar{y})
 +\mathcal{F}_u^-(\bar{x},\bar{y})$, where
 \begin{align*}
 \mathcal{F}_u^+(\bar{x},\bar{y})\ :=\ \inf_{\sigma_{\bar{x}\bar{y}}}
 \int_{\frac{d_{\infty}(\bar{x},\bar{y})}{2}}^{d_{\infty}(\bar{x},\bar{y})}
 u(\sigma_{\bar{x}\bar{y}}(t))\ \text{d}t,\quad \text{and}\quad 
 \mathcal{F}_u^-(\bar{x},\bar{y})\ :=\
 \inf_{\sigma_{\bar{x}\bar{y}}}
 \int_0^{\frac{d_{\infty}(\bar{x},\bar{y})}{2}}u(\sigma_{\bar{x}\bar{y}}(t))\
 \text{d}t.
 \end{align*}
 Since $\mathcal{F}_u^+(\bar{x},\bar{y})=\mathcal{F}_u^-(\bar{y},\bar{x})$, by
 Fubini's theorem we have 
 \begin{align*}
 \int_{B_{d_{\infty}}(\bar{x}_0,4r)} \int_{B_{d_{\infty}}(\bar{y}_0,4r)}
 \mathcal{F}_u^+(\bar{x},\bar{y})\ \text{d}\mathcal{H}^k(\bar{y})
 \text{d}\mathcal{H}^k(\bar{x})\ =\ \int_{B_{d_{\infty}}(\bar{x}_0,4r)} 
 \int_{B_{d_{\infty}}(\bar{y}_0,4r)} \mathcal{F}_u^-(\bar{x},\bar{y})\
 \text{d}\mathcal{H}^k(\bar{y})\text{d}\mathcal{H}^k(\bar{x}),
 \end{align*}
 and we only need to establish the estimate for $\mathcal{F}_u^+$. 

 We now fix any $\bar{x}\in B_{d_{\infty}}(\bar{x}_0,r)\cap \mathcal{R}$. By the
 triangle inequality, the unit-speed minimal geodesic $\sigma_{\bar{x}\bar{y}}$
 connecting $\bar{x}$ to any $\bar{y}\in  B_{d_{\infty}}(\bar{y}_0,r)\cap
 \mathcal{R}$ is entirely contained in $B_{d_{\infty}}(\bar{y},4r)$. Moerover,
 as we have discussed before in (\ref{eqn: good_nbhd}),
 $\sigma_{\bar{x}\bar{y}}$ is a straight line segment which has a Euclidean
 tubular neighborhood. Even though it is not possible to obtain a uniform size
 of the neighborhood, the key observation is that however small the
 neighborhood is, infinitesimally we can decompose the volume form
 $\text{d}\mathcal{H}^k$ at $\sigma_{\bar{x}\bar{y}}(s)$ into polar coordinate
 as $s^{k-1}\text{d}\theta(\vec{v}) \text{d}s$, where $\text{d}\theta$ is
 the volume form on $\mathbb{S}^{k-1}$, the standard $k-1$ sphere, and
 $\vec{v}\in \mathbb{S}^{k-1}$ is essentially
 $\dot{\sigma}_{\bar{x}\bar{y}}(0)$ --- since $B_{\bar{x}}\equiv
 \mathbb{B}^k(\bar{r}_{\bar{x}})$, all unit-speed minimal geodesics emanating
 from $\bar{x}$ is parametrized by $\mathbb{S}^{k-1}$. Adapting to the polar
 structure at $\bar{x}$, we will let $\sigma_{\bar{x},\vec{v}}$ denote the
 minimal geodesic emanating from $\bar{x}$ with initial direction $\vec{v}\in
 \mathbb{S}^{k-1}$. Consequently, we have
 \begin{align*}
 \mathcal{F}_u^+\left(\sigma_{\bar{x},\vec{v}}\left(\frac{t}{2}\right),
 \sigma_{\bar{x},\vec{v}}(t)\right)\
 \text{d}\mathcal{H}^k(\sigma_{\bar{x},\vec{v}}(t))\ \le\
 &\left(\int_{\frac{t}{2}}^tu(\sigma_{\bar{x},\vec{v}}(s))\ \text{d}s\right)\
 t^{k-1}\text{d}\theta(\vec{v})\text{d}t\\
 \le\ &2^{k-1}\left(\int_{\frac{t}{2}}^t u(\sigma_{\bar{x},\vec{v}}(s))s^{k-1}
 \text{d}\theta(\vec{v})\text{d}s\right)\ \text{d}t.
  \end{align*} 
Moreover, we will let  $\bar{t}_r(\bar{x},\vec{v}):=\sup\left\{t>0:\
\sigma_{\bar{x},\vec{v}}(t)\in \mathcal{R}\cap B_{d_{\infty}}
(\bar{y}_0,4r)\right\}$ --- notice that $\sigma_{\bar{x},\vec{v}}$ is always
defined at least up to $t=\bar{r}_{\bar{x}}$ and $\bar{x}\in
B_{d_{\infty}}(\bar{y}_0,4r)$, so $\bar{t}_r(\bar{x},\vec{v})>0$. On the other
hand, by the triangle inequality, it is clear that $\bar{t}_r (\bar{x},\vec{v})
\le 6r$ for any $\vec{v}\in \mathbb{S}^{k-1}$. Now since $\vec{v}\in
\mathbb{S}^{k-1}$ exhausts all possible directions of minimal geodesics
emanating from $\bar{x}$ and reaching to
$\bar{y}\in B_{d_{\infty}}(\bar{y}_0,r)\cap \mathcal{R}$, we can integrate the
above inequality with respect to $\bar{y}\in B_{d_{\infty}} (\bar{y}_0,r)\cap
\mathcal{R}$ and apply Fubini's theorem to see that
\begin{align*}
\int_{B_{d_{\infty}}(\bar{y}_0,r)}\mathcal{F}_u^+(\bar{x},\bar{y})\
\text{d}\mathcal{H}^k(\bar{y})\ \le\
&\int_{\mathbb{S}^{k-1}}\int_0^{\bar{t}_r(\bar{x},\vec{v})}
\mathcal{F}_u^+\left(\sigma_{\bar{x},\vec{v}}\left(\frac{t}{2}\right),
\sigma_{\bar{x},\vec{v}}(t)\right)\ \text{d}\mathcal{H}^k
(\sigma_{\bar{x},\vec{v}}(t))\\
\le\ &2^{k-1}\int_0^{6r}\int_{\mathbb{S}^{k-1}}
\left(\int_{0}^{\bar{t}_r(\bar{x},\vec{v})}
u(\sigma_{\bar{x}, \vec{v}}(s))s^{k-1} \text{d}\theta(\vec{v})\text{d}s\right)\
\text{d}t\\
\le\ &2^{k+2}r\int_{B_{d_{\infty}}(\bar{y}_0,4r)} u\ \text{d}\mathcal{H}^k.
\end{align*}
Here on the left-hand side we are essentially integrating regular $\bar{y}$, but
as $\mathcal{R}$ is of full Hausdorff measure, it is the same as integrating
$\bar{y}$ all over $B_{d_{\infty}}(\bar{y}_0,r)$. Integrating the last inequlity
with respect to $\bar{x}\in B_{d_{\infty}}(\bar{x}_0,r)\cap \mathcal{R}$, we
have
\begin{align*}
\int_{B_{d_{\infty}}(\bar{x}_0,r)}
\int_{B_{d_{\infty}}(\bar{y}_0,r)}\mathcal{F}_u^+(\bar{x},\bar{y})\
\text{d}\mathcal{H}^k(\bar{y})\text{d}\mathcal{H}^k(\bar{x})\ \le\
2^{k+2}r\ \mathcal{H}^k\left(B_{d_{\infty}}(\bar{y}_0,4r)\right)
\int_{B_{d_{\infty}}(\bar{y}_0,4r)} u\ \text{d}\mathcal{H}^k,
\end{align*}
as $B_{d_{\infty}}(\bar{x}_0, r)\subset B_{d_{\infty}}(\bar{y}_0, 4r)$. Adding
its symmetric part we obtain 
\begin{align*}
\int_{B_{d_{\infty}}(\bar{x}_0,r)} \int_{B_{d_{\infty}}(\bar{y}_0,r)}
\mathcal{F}_u(\bar{x},\bar{y})\
\text{d}\mathcal{H}^k(\bar{y})\text{d}\mathcal{H}^k(\bar{x})\ \le\ 2^{k+3}r\
\mathcal{H}^k\left(B_{d_{\infty}}(\bar{y}_0,4r)\right)
\int_{B_{d_{\infty}}(\bar{y}_0,4r)} u\ \text{d}\mathcal{H}^k,
\end{align*}
which is the desired inequality.
 \end{proof}
 
In order to show that $B_{d_{\infty}}(\bar{p}_{\infty},10)\equiv
\mathbb{B}^k(10)$, we will rely on the recent developments in the theory of
$RCD$ spaces, notably \cite{DPG18}, together with a volume consideration. To
this end, we equip $W_{\infty}:=\overline{B_{d_{\infty}}
(\bar{p}_{\infty},20)}$ with the $k$-dimensional Hausdorff measure
$\mathcal{H}^k$, and prove the following
\begin{claim}\label{clm: RCD}
 The metric measure space $(W_{\infty},d_{\infty}, \mathcal{H}^k)$ is a compact,
 connected and non-collapsing $RCD(0,k)$ space.
\end{claim}
Although the metric measure space $(W_{\infty},d_{\infty},\nu)$, as the
\emph{metric-measure limit} of the closed geodesic balls
$(\overline{B_{\bar{g}_i}(\bar{p}_i, 20)}, \bar{g}_i,
\left|B_{\bar{g}_i}(\bar{p}_i,1)\right|^{-1}\dvol_{\bar{g}_i})$, is already an
$RCD(m-1,m)$ space by the standard theory, it is not immediate that the
\emph{metric limit} $(W_{\infty},d_{\infty})$, when equipped with the
$k$-dimensional Hausdorff measure $\mathcal{H}^k$, is an $RCD(0,k)$ space.
Before proving the claim, let us briefly recall the definition of $RCD(0,k)$
spaces proposed in \cite{AGS14}; see also \cite{Gigli18, Honda20}. 
For a complete, Hausdorff and separable metric measure space $(X,d_X,\mathfrak{m})$
the Cheeger energy is defined for any $f\in L^2(X)$ as
\begin{align*}
Ch(f)\ :=\ \inf\left\{\liminf_{j\to \infty}\frac{1}{2}\int_X(lip\ f_j)^2\
\text{d}\mathfrak{m}:\ f_j\in Lip_b(X)\cap L^2(X),\
\lim_{j\to\infty}\|f_j-f\|_{L^2(X)}=0\right\},
\end{align*} 
where $Lip_b(X)$ denote the collection of locally bounded Lipschitz functions on
$X$, and for any such $f$, $lip\ f$ is defined at any $x\in X$ as
$lip\ f(x) :=\liminf_{r\searrow 0}\sup_{y\in
B_{d_X}(x,r)}|f(x)-f(y)|d_X(x,y)^{-1}$.
The Banach space $W^{1,2}(X)$ is defined as $\left\{f\in L^2(X):\
Ch(f)<\infty\right\}$. 
\begin{definition}\label{defn: RCD0k}
We say that a complete, Hausdorff and separable metric measure space
$(X,d_X,\mathfrak{m})$ an $RCD(0,k)$ space if 
\begin{enumerate}
  \item there is a \emph{carr\'e du champ} $\Gamma: W^{1,2}(X)\times
  W^{1,2}(X) \to L^1(X)$ such that $\forall f\in W^{1,2}(X)$, if
  $\Gamma(f,f)\le 1$ $\mathfrak{m}$-a.e. on $X$, then it has $1$-Lipschitz
  representative;
  \item Infinitesimally Hilbertian: the Cheeger energy can be written as
  $2Ch(f)=\int_X\Gamma(f,f)\ \text{d}\mathfrak{m}$ for any $f\in W^{1,2}(X)$;
  \item there exists $x\in X$ such that $\mathfrak{m}(B_{d_X}(x,r))\le Ce^{Cr}$
  for some fixed $C>1$ and any $r>0$;
  \item Bakry-\'Emery inequality: $\forall f\in Dom(\Delta)$ with $\Delta f\in
  W^{1,2}(X)$, we have
  \begin{align*}
  \int_X\Gamma(f,f)\Delta\varphi\ \text{d}\mathfrak{m}\ \ge\
  2\int_X\left(\frac{1}{k}(\Delta f)^2+\Gamma(f,\Delta f)\right)\varphi\
  \text{d}\mathfrak{m},
  \end{align*}
  whenever $\varphi \in Dom(\Delta)\cap L^{\infty}(X)$ satisfies $\varphi\ge 0$
  and $\Delta \varphi \in L^{\infty}(X)$.
\end{enumerate}
Moreover, if $\mathfrak{m}=\mathcal{H}^k$, the $k$-dimensional Hausdorff
measure, then we say that $(X,d_X,\mathfrak{m})$ is a \emph{non-collapsing}
$RCD(0,K)$ space, denoted by $ncRCD(0,K)$ space; see \cite[Definition
1.1]{DPG18}.
\end{definition} 
Here we notice that $\Delta$ denotes the Laplace operator on $X$: for some
$f\in Dom(\Delta)\subset W^{1,2}(X)$, we say that $u\in L^2(X)$ satisfies
$\Delta f=u$ if
\begin{align*}
\forall \varphi\in W^{1,2}(X),\quad \int_{X}\Gamma(f,\varphi)\
\text{d}\mathfrak{m}\ =\ -\int_Xu\varphi\ \text{d}\mathfrak{m}.
\end{align*}
Notice that here $f$ is a critical point of the functional
$\mathcal{E}_u(f):=\int_X\frac{1}{2}\Gamma(f,f)+fu\ \text{d}\mathcal{H}^k$. When
$X$ has a boundary in a given context, we require the energy functional
$\mathcal{E}_u$ to only consider $W^{1,2}(W_{\infty})$ functions subject to a
given boundary condition $f|_{\partial X}$. In this situation, we require the
test function $\varphi$ to vanishi on $\partial X$, both in the definition of
$\Delta$ and in Condition (4) of Definition~\ref{defn: RCD0k}.

\begin{proof}[Proof of Claim~\ref{clm: RCD}]
If $f\in Lip_b(W_{\infty})$, then clearly $\nabla f$ is defined
$\mathcal{H}^k$-a.e. in $\mathcal{R}_{20}$. Consequently, we can define the
\emph{carr\'e du champ} $\Gamma$ for  $f_1,f_2\in Lip_b(W_{\infty})$ as
$\Gamma(f_1,f_2):=\left\langle \nabla f_1,\nabla f_2\right\rangle$
for $\mathcal{H}^{k}$-a.e. point in $W_{\infty}$. Here
$\left\langle-,-\right\rangle$ denotes the Euclidean inner product. Notice that
here the orbifold nature of $Z$ is crucial --- any regular point of
$W_{\infty}$ is the center of a small enough geodesic ball isometric to a
$k$-dimensional Euclidean ball.

Now for any $f_1,f_2\in W^{1,2}(W_{\infty})$, there is a sequence
$\{f_{aj}\}\subset Lip_b(W_{\infty})$ such that for $a=1,2$,
\begin{align*}
\lim_{j\to\infty}
\left\|f_{aj}-f_a\right\|_{L^2(W_{\infty})}\ =\ 0\quad \text{and}\quad
\lim_{j\to \infty} \left\|\nabla f_{aj}\right\|^2_{L^2(W_{\infty})}\ =\
2Ch(f_a)\ <\ \infty.
\end{align*}
Passing to sub-sequences, we can assume that
$f_{aj}\xrightarrow{\mathcal{H}^k\text{-a.e.}} f_{a}$ for $a=1,2$, and thus we
can define $\Gamma(f_1,f_2):=\lim_{j\to \infty} \left\langle \nabla f_{1j}, \nabla
f_{2j}\right\rangle$ for $\mathcal{H}^{k}$-a.e. point in $W_{\infty}$. (Here we
can truncate $|\nabla f_{aj}|$ by $n\in\mathbb{N}$, define a limit for each $n$
and take a diagonal limit as $n\to \infty$.) Clear, $\Gamma$ maps into
$L^1(W_{\infty})$ by the dominated convergence theorem: we have
$\left|\Gamma(f_{1j},f_{2j})\right|\le |\nabla f_{1j}||\nabla
f_{2j}|\in L^1(W_{\infty})$ by H\"older's inequality, and consequently
\begin{align*}
\int_{W_{\infty}}\left|\Gamma(f_1,f_2)\right|\
\text{d}\mathcal{H}^k\ =\ &\int_{W_{\infty}}
\lim_{j\to\infty}\left|\Gamma(f_{1j},f_{2j})\right|\ \text{d}\mathcal{H}^k\\
\le\ &\liminf_{j\to\infty}\left\|\nabla f_{1j}\right\|_{L^2(W_{\infty})}
\left\|\nabla f_{2j}\right\|_{L^2(W_{\infty})}\\
\le\ &2\sqrt{Ch(f_1)Ch(f_2)}\ <\ \infty.
 \end{align*} 
By the same reasoning, $\Gamma$ makes the Cheeger energy a quadratic form: if
$f_j\xrightarrow{L^2}f\in W^{1,2}(W_{\infty})$ and
$\left\|\nabla f_j\right\|_{L^2(W_{\infty})}^2 \to 2Ch(f)$, then $\left\|\nabla
f_{j}\right\|_{L^2(W_{\infty})}\le 2Ch(f)^{\frac{1}{2}}<\infty$ for all $j$
sufficiently large, and after possibly passing to a sub-sequence, by the
dominated convergence theorem we have
\begin{align*}
2Ch(f)\ =\ &\lim_{j\to\infty}\left\|\nabla
f_{j}\right\|_{L^2(W_{\infty})}^2\\
=\ &\int_{W_{\infty}}\lim_{j\to \infty}\left\langle \nabla f_{j},
\nabla f_{j}\right\rangle\ \text{d}\mathcal{H}^k\\
=\ &\int_{W_{\infty}}\Gamma(f,f)\ \text{d}\mathcal{H}^k.
\end{align*}

While we have defined the \emph{carr\'e du champ} $\Gamma$ and checked the
infinitesimally Hilbertian property (Condition (2)), given $f\in
W^{1,2}(W_{\infty})$ with $\Gamma(f,f)\le 1$, we are yet to find a
$1$-Lipschitz function $\bar{f}$ such that $\bar{f}=f$ for $\mathcal{H}^k$-a.e.
point in $W_{\infty}$. We now assume to have a sequence $\{f_j\}\subset
Lip_b(W_{\infty})$ such that 
$f_j\xrightarrow{\mathcal{H}^k\text{-a.e.}}f$ and $\left\|\nabla
f_j\right\|^2_{L^2(W_{\infty})}\to 2Ch(f)$, then $\frac{1}{2}\left|\nabla
f_j\right|^2\xrightarrow{\mathcal{H}^k\text{-a.e.}}\Gamma(f,f)\le 1$ after
possibly passing to a sub-sequence, still denoted by $\{f_j\}$. In order to find a
Lipschitz representative of $f$, we need to check the uniform equicontinuity of
$\{f_j\}$ --- notice that $\left|\nabla f_j\right|$ only converges to a bounded
function $\mathcal{H}^k$-\emph{almost everywhere}, and we do not have a uniform
bound for $\left\|\nabla f_j\right\|_{L^{\infty}(W_{\infty})}$. Therefore, we
will need to rely on the segment inequality (\ref{eqn: segment}) which controls
the modulus of continuity by $\left\|\nabla f_j\right\|_{\bar{L}^2_{loc}
(W_{\infty})}$ (the local $L^2$-average).

Now for any $\bar{p},\bar{q}\in W_{\infty}$ with
$d_{\infty}(\bar{p},\bar{q})=r\le 1$, for any $\varepsilon>0$ there are
$\bar{p}_{\varepsilon},\bar{q}_{\varepsilon}\in \mathcal{R}_{20}$ such that
$d_{\infty}(\bar{p}, \bar{p}')+d_{\infty}(\bar{q},\bar{q}')<\varepsilon r$.
Fixing a minimal geodesic $\sigma_{\varepsilon}:[0,1]\to M$ connecting
$\bar{p}_{\varepsilon}$ and $\bar{q}_{\varepsilon}$, we have a positive radius
$\bar{r}_{\sigma_{\varepsilon}}$ such that the
$\bar{r}_{\sigma_{\varepsilon}}$-tubular neighborhood 
of $\sigma_{\varepsilon}$ is isometric to an Euclidean domain; see (\ref{eqn:
good_nbhd}). We now fix $n\in \mathbb{N}$ large enough so that
$10n^{-1}r<\bar{r}_{\sigma_{\varepsilon}}$. For
$\bar{x}_s:=\sigma_{\varepsilon}(s)$ with $s=1,\ldots,n$, we define
\begin{align*}
A^{n}_{s-1}\ :=\ \left\{\bar{x}\in
B_{d_{\infty}}(\bar{x}_{s-1},n^{-1}r)\cap \mathcal{R}:\
\fint_{B_{d_{\infty}}(\bar{x}_{s}, n^{-1}r)}
\mathcal{F}_{\left|\nabla f_j\right|}(\bar{x},\bar{y})\
\text{d}\mathcal{H}^k(\bar{y})\ \le\
32^{k+1}r\fint_{B_{d_{\infty}}(\bar{x}_s,4n^{-1}r)} \left|\nabla
f_j\right|\right\},
\end{align*}
and for any $\bar{x}\in A^{n}_{s-1}$, we define 
\begin{align*}
B^{n}_s(\bar{x})\ :=\ \left\{\bar{y}\in
B_{d_{\infty}}(\bar{x}_{s},n^{-1}r)\cap \mathcal{R}:\
\mathcal{F}_{\left|\nabla f_j\right|}(\bar{x},\bar{y})\ \le\
32^{k+2}r\fint_{B_{d_{\infty}}(\bar{x}_s,4n^{-1}r)} \left|\nabla
f_j\right|\right\}.
\end{align*}
We also define $A^{n}_n := B_{d_{\infty}}(\bar{x}_n,n^{-1}r)\cap \mathcal{R}$
and $B^{n}_0(\bar{x}):= B_{d_{\infty}}(\bar{x}_0, n^{-1}r)\cap \mathcal{R}$ for
any $\bar{x}\in A^{n}_0$.

 Now for each $s=1,\ldots,n$, we apply Claim~\ref{clm: segment} to $|\nabla
 f_j|$ on $B_{d_{\infty}}(\bar{x}_{s-1},n^{-1}r)\cup
 B_{d_{\infty}}(\bar{x}_{s-1},n^{-1}r)$ and see that
\begin{align}\label{eqn: segment_s}
\fint_{B_{d_{\infty}}(\bar{x}_{s-1},n^{-1}r)}
\fint_{B_{d_{\infty}}(\bar{x}_{s},n^{-1}r)} \mathcal{F}_{\left|\nabla
f_j\right|}\ \text{d}\mathcal{H}^k \text{d}\mathcal{H}^k\ \le\
2^{5k+3}n^{-1}r\fint_{B_{d_{\infty}}(\bar{x}_s,4n^{-1}r)} \left|\nabla
f_j\right|\ \text{d}\mathcal{H}^k.
\end{align}
The right-hand side constant is obtained as $B_{d_{\infty}}
(\bar{x}_s, 10n^{-1}r)\subset B_{d_{\infty}}
\left(\sigma_{\varepsilon}([0,1]), \bar{r}_{\sigma_{\varepsilon}}\right)$, and
thus
\begin{align*}
\mathcal{H}^k\left( B_{d_{\infty}}(\sigma_{\varepsilon}
(\bar{x}_s),n^{-1}r)\right)\ =\ \omega_k(n^{-1}r)^k\quad \text{and}\quad 
\mathcal{H}^k\left(B_{d_{\infty}}(\sigma_{\varepsilon} (\bar{x}_s),
4n^{-1}r)\right)\ =\ \omega_k(4n^{-1}r)^k
\end{align*}
for each $s=0,1,\ldots,n$. Now by Chebyshev's inequality we see for each
$s=1,\ldots,n$ that
\begin{align*}
\mathcal{H}^k\left(A^{n}_{s-1}\right)\ \ge\
\frac{3}{4}\mathcal{H}^k\left(B_{d_{\infty}} (\bar{x}_{s-1}, n^{-1}r)\right),
\quad \text{and}\quad \mathcal{H}^k\left(B^{n}_{s}(\bar{x})\right)\ \ge\
\frac{31}{32}\mathcal{H}^k\left(B_{d_{\infty}} (\bar{x}_{s}, n^{-1}r)\right)
\end{align*}
whenever $\bar{x}\in A^{n}_{s-1}$. Especially, $A^{n}_s\cap
B^{n}_s(\bar{x})\not=\emptyset$ for any $s=1,\ldots,n$ and any $\bar{x}\in
A^{n}_{s-1}$. On the other hand, by the assumption that $\left|\nabla
f_j\right| \xrightarrow{\mathcal{H}^k\text{-a.e.}} \Gamma(f,f)\le 1$ when $j\to
\infty$ we see for all $j$ large enough and any $s=1,\ldots,n$ that
\begin{align*}
\fint_{B_{d_{\infty}}(\bar{x}(s), 4n^{-1}r)}\left|\nabla f_j\right|^2\
\text{d}\mathcal{H}^k\ \le\ 2.
\end{align*}
Combining this inequality with (\ref{eqn: segment_s}) and starting with any
$\bar{x}_0'\in A^{n}_0$ fixed, we could then inductively find minimal
geodesic segments $\sigma_s$ connecting $\bar{x}_{s}'\in A^{n}_{s}\cap
B^{n}_{s}(\bar{x}_{s-1}')$, such that 
\begin{align*}
\forall s=1,\ldots,n,\quad \left|f_j(\bar{x}_s')-f_j(\bar{x}_{s-1}')\right|\
=\ \int_{\sigma_s}\left\langle \nabla f_j,\dot{\sigma}_s\right\rangle\ 
\le\ 32^{k+4}n^{-1}r.
\end{align*}
Here we notice that each $\sigma_s$ is contained in $\mathcal{R}_{40}$ since
$\sigma_s$ maps into $B_{d_{\infty}}(\bar{x}_s',8n^{-1}r)
\subset B_{d_{\infty}}(\bar{p}_{\infty}, 23)$ for $n\ge 8$. Adding these
inequalities up with respect to $s$, we see that
\begin{align}\label{eqn: fj_estimate}
\left|f_j(\bar{x}_n')-f_j(\bar{x}_0')\right|\ \le\ 32^{k+4}r.
\end{align}
Since $d_{\infty}(\bar{p},\bar{x}_0')+d_{\infty}(\bar{q},\bar{x}_n')
<(2n^{-1}+\varepsilon)r$, by the continuity of $f_j$ and the independence
of the estimate (\ref{eqn: fj_estimate}) on $n$ and $\varepsilon$, we can let
$n\to \infty$, then send $\varepsilon\to 0$, and see that
\begin{align*}
\left|f_j(\bar{p})-f_j(\bar{q})\right|\ \le\ 
32^{k+4}d_{\infty}(\bar{p},\bar{q}).
\end{align*}
 As the uniform constant $32^{k+4}$ is clearly independent of $j$, this proves
 the equicontinuity of the family $\{f_j\}$, and the Arzel\`a-Ascoli theorem
 allows us to extract a continuous limit function $\bar{f}$. Clearly,
 $\bar{f}=f$ holds $\mathcal{H}^k$-a.e.; on the other hand, since
 $\Gamma(f,f)\le 1$, the infinitesimally Hilbertian property ensures that the
 \emph{upper gradient} $|Df|$ of $f$ is $\mathcal{H}^k$-a.e. bounded above by
 $1$, and by the locality of the upper gradient (see \cite[(2.7)]{Gigli18}), we
 see that $\left|D\bar{f}\right|\le 1$ holds $\mathcal{H}^k$ almost everywhere,
 implying that $\bar{f}$ is Lipschitz with $lip\ f\le 1$. We have now finished
 verifying Condition (1).

By the Claim~\ref{clm: fix_center}, we see that Condition (3) is readily 
checked as  
\begin{align*}
\mathcal{H}^k\left(B_{d_{\infty}}(\bar{p}_{\infty},r)\right)\ 
=\ |H|\left|B_{d_Z}(z,r)\right|\ =\ \omega_kr^k\ <\ \omega_ke^r.
\end{align*} 

 We now check the Bakry-\'Emery inequality for $f\in  Dom(\Delta)$,
 i.e. Condition (4). By the infinitesimally Hilbertian property, we know that
 $\Delta$ is a linear operator, and that it is a local operator. Consequently,
 for any $\bar{x}\in \mathcal{R}_{20}$, let $B_{\bar{x}}$ be the $k$-Euclidean
 neighborhood of $\bar{x}$, then $(\Delta f)|_{B_{\bar{x}}}=\Delta
 (f|_{B_{\bar{x}}})$ $\mathcal{H}^k$-a.e. whenever $\Delta f\in
 W^{1,2}(W_{\infty})$. But in this case, we can handle the Laplace operator as
 the Euclidean one: as $(\Delta f)|_{B_{\bar{x}}}\in W^{1,2}(B_{\bar{x}})$, by
 the elliptic regularity in the $k$-Euclidean space we see that
 $\|f\|_{C^2(\frac{1}{2}B_{\bar{x}})}<\infty$ where $\frac{1}{2}B_{\bar{x}}$ is
 the concentric ball as $B_{\bar{x}}$ but with half the radius. Consequently, as
 $\Delta f\in W^{1,2}(W_{\infty})$, we can estimate at $\mathcal{H}^k$-a.e. 
 regular point $\bar{x}$ that 
 \begin{align}\label{eqn: define_Gammaff}
 \begin{split}
\Delta \Gamma(f,f)(\bar{x})\ =\
&2\left|Hess_f\right|^2(\bar{x})+2\Gamma( f, \Delta f) (\bar{x})\\
\ge\ &\frac{2}{k}(\Delta f)^2(\bar{x})+2\Gamma(f,\Delta f)(\bar{x}).
\end{split}
 \end{align}
 Especially, we see that $\Delta \Gamma(f,f)$ is well-defined $\mathcal{H}^k$
 almost everywhere on $\mathcal{R}_{20}$. On the other hand, since
 $\mathcal{R}_{20}$ is an open set, and especially $\bar{r}_{\bar{x}}=
 d_{\infty} (\bar{x},\mathcal{S})$ for any $\bar{x}\in \mathcal{R}_{20}$, we
 could mollify $f$ to find $f_j\in C^3(\mathcal{R}_{20})$ so that
 $f_j|_{\mathcal{S}}=0$ and $f_j\xrightarrow{\mathcal{H}^k\text{-a.e.}} f$ as
 $j\to \infty$. Consequently, this also ensures
\begin{align}\label{eqn: mollification1}
 \frac{1}{k}(\Delta f_j)^2+\Gamma(f_j,\Delta f_j)\ \longrightarrow\
 \frac{1}{k}(\Delta f)^2+\Gamma(f,\Delta f)\quad \mathcal{H}^k\text{-a.e.}\
 \text{as}\ j\to \infty,
\end{align}
since by (\ref{eqn: define_Gammaff}) the limit is $\mathcal{H}^k$-a.e. well
defined on $\mathcal{R}_{20}$; and similarly 
\begin{align}\label{eqn: mollification2}
\Gamma(f_j,f_j)\ \longrightarrow\ \Gamma(f,f)\quad \mathcal{H}^k\text{-a.e.}\
\text{as}\ j\to \infty.
\end{align}
 
 Now for for any non-negative $\varphi\in  Dom(\Delta)\cap
 L^{\infty}(W_{\infty})$ with $\Delta \varphi\in L^{\infty}(W_{\infty})$ such
 that $\varphi=0$ on $\partial W_{\infty}=\left\{\bar{x}\in Y_{\infty}:\
 d_{\infty}(\bar{x},\bar{p}_{\infty})=20\right\}$, we have
\begin{align*}
\int_{W_{\infty}} \Gamma(f,f)\Delta \varphi\ \text{d}\mathcal{H}^k =\
&\lim_{j\to\infty}\int_{W_{\infty}}\Gamma(f_i,f_i)\Delta \varphi\
\text{d}\mathcal{H}^k\\
=\ &\lim_{j\to\infty}\int_{W_{\infty}} \Delta \Gamma(f_i,f_i)\varphi\
\text{d}\mathcal{H}^k\\
\ge\ &\lim_{j\to\infty}2\int_{W_{\infty}}\left(\frac{1}{k}(\Delta
f_i)^2+\Gamma(f_i,\Delta f_i)\right)\varphi\ \text{d}\mathcal{H}^k\\
\ge\ &2\int_{W_{\infty}}\left(\frac{1}{k}(\Delta f)^2 +\Gamma(f,\Delta
f)\right)\varphi\ \text{d}\mathcal{H}^k,
\end{align*}
where the first equality holds because of (\ref{eqn: mollification2}) and the
dominated convergence theorem, while the last inequality holds due to
(\ref{eqn: mollification1}) and Fatou's lemma. Condition (4) is now verified. 

With all conditions in Definition~\ref{defn: RCD0k} checked, and since the
measure $\mathcal{H}^k$ is nothing but the $k$-dimensional Hausdorff measure,
we conclude that $(Y_{\infty},d_{\infty},\mathcal{H}^k)$ is an $ncRCD(0,k)$
space.
\end{proof}

Moreover, since $B_{d_{\infty}}(\bar{p}_{\infty}, 20)\slash
H=W_{\infty}\slash H\equiv B_{d_Z}(z,20)\subset Z$, we see that
\begin{align*}
\mathcal{H}^k\left(B_{\bar{d}_{\infty}}(\bar{p}_{\infty},20)\right)\ =\ 
|H|\left|B_{d_Z}(z,20)\right|\ =\ 
\left|\mathbb{B}^k(20)\right|.
\end{align*}
Therefore, applying \cite[Theorem 1.5]{DPG18} to the  $ncRCD(0,k)$
space $(W_{\infty},d_{\infty}, \mathcal{H}^k)$, we have  
\begin{align*}
B_{d_{\infty}}(\bar{p}_{\infty},10)\ \equiv\ \mathbb{B}^k(10).
\end{align*}
This concludes the proof of the lemma.
\end{proof}
Now since $\lim_{i\to\infty}d_{GH}\left(B_{\bar{g}_i}(\bar{p}_i,10),
B_{d_{\infty}}(\bar{p}_{\infty}, 10)\right)= 0$, Lemma~\ref{lem: isometry}
produces a contradiction to (\ref{eqn: contradiction_delta_AE}) as $i\to
\infty$, and this contradiction shows that that for any $i$ sufficiently large,
\begin{align}\label{eqn: Y_GH}
10R^{-1}\delta'_i\ :=\ d_{GH}\left(B_{\bar{g}_i}(\bar{p}_i,10),
\mathbb{B}^k(10)\right)\ <\ \delta_{AE}\left(2^{-1}\varepsilon'\right).
\end{align}
Without loss of generality, we may assume that $\delta'_i\ge  \delta_i$.

 On the other hand, since for each $i$ the map $\pi_{i,0}:
 \pi_i^{-1}(B_{g_i'}(p_i,25)) \to Y_i$ is a normal covering with deck
 transformation group $\ker \phi_i$, so is its restriction to
 $\pi_{i,0}^{-1}(B_{\bar{g}_i}(\bar{p}_i,10) \to B_{\bar{g}_i}(\bar{p}_i,10)$.
 Since the pseudo-local fundamental group $\tilde{\Gamma}_{\delta_i'}(p_i)\le
\pi_1(B_{g_i}(K_i,R),p_i)$ is characterized by
\begin{align*}
\tilde{\Gamma}_{\delta_i'}(p_i)\ \cong\ 
\left\langle \gamma\in \pi_1(B_{g_i}(K_i,R),p_i):\
d_{\pi^{\ast}g_i}(\gamma.\tilde{p}_i,\tilde{p}_i)<2\delta_i'\right\rangle,
\end{align*}
 we have the following identity of almost nilpotent groups: 
\begin{align}\label{eqn: widehat_G}
\widehat{G}_{\delta_i'}(\bar{p}_i)\ =\ 
\left\langle \gamma\in \ker \phi_i:\ d_{\pi^{\ast}g'_i}(\gamma.\tilde{p}_i,
\tilde{p}_i)<20R^{-1}\delta_i' \right\rangle\ =\ 
\tilde{\Gamma}_{\delta_i'}(p_i)\cap \ker \phi_i.
\end{align}
Moreover, since
\begin{align*}
\widehat{G}_{\delta_i'}(\bar{p}_i)\ =\
\tilde{\Gamma}_{\delta_i'}(p_i)\cap \ker \phi_i\ \trianglelefteq\
\tilde{\Gamma}_{\delta_i'}(p_i)\quad \text{and}\quad
[\pi_1(B_{g_i}(K_i,R),p_i):\ker \phi_i]\ =\ |G|\ <\ \infty,
\end{align*} 
we must have $\left[\tilde{\Gamma}_{\delta_i'}(p_i):
\widehat{G}_{\delta_i'}(\bar{p}_i)\right]\le l$, and
thus $\rank\ \widehat{G}_{\delta_i'}(\bar{p}_i) =\rank\
\tilde{\Gamma}_{\delta_i'}(p_i)$ --- the almost nilpotency of
$\widehat{G}_{\delta_i'}(\bar{p}_i)$ is guaranteed by its
normality within $\tilde{\Gamma}_{\delta_i'}(p_i)$, which is almost nilpotent
for $\delta_i$ sufficiently small (see \cite{KW11, NaberZhang}).

Since $\delta_i\le \delta_i'$, we have $\tilde{\Gamma}_{\delta_i}(p_i)\le
\tilde{\Gamma}_{\delta_i'}(p_i)$ by the definition of the pseudo-local
fundamental group, and the Assumption (3), i.e.
$\rank\ \tilde{\Gamma}_{\delta_i}(p_i)=m-k$, then implies that $\rank\
\tilde{\Gamma}_{\delta_i'}(p_i)\ge m-k$. Consequently, we have 
\begin{align}\label{eqn: G[p]_rank}
\rank\ \widehat{G}_{\delta_i'}(\bar{p}_i)\ \ge\ m-k.
\end{align}
Now applying Lemma~\ref{lem: almost_Euclidean} to the normal covering
$\pi_{i,0}^{-1}(B_{\bar{g}_i}(\bar{p}_i,10))\to B_{\bar{g}_i}(\bar{p}_i,10)$
with deck transformation group $\ker \phi_i$, by (\ref{eqn: Y_GH}), (\ref{eqn:
widehat_G}) and (\ref{eqn: G[p]_rank}), we can find some scale
$\tilde{r}:=10^{-1}r_{AE}R\in (0,\frac{R}{10})$, such that for all $i$ large
enough, the isoperimetric constant has lower bound
\begin{align*}
I_{\pi^{\ast}g_i}\left(B_{\pi^{\ast}g_i}(\tilde{p}_i,\tilde{r})\right)\ >\
(1-2^{-1}\varepsilon')I_m,
\end{align*}
contradicting (\ref{eqn: contradiction_epsilon}). As discussed at the very
beginning of the section, this contradiction establishes Theorem~\ref{thm:
main2}. Notice that the \emph{a priori} dependence of $\delta_O$ is on $Z$, but
as $Z\equiv \mathbb{B}^k(25)\slash G$, and $G<O(k)$ has no more than
$l$ elements, there are only finitely many such finite groups, and the
dependence of $\delta_O$ is then reduced to $l$, the order of the local orbifold
group (besides the dimension $k\le m$).

\section{Infranil fiber bundles over controlled orbifolds}
In this section, we apply the local Ricci flow smoothing result to prove
Theorem~\ref{thm: main3}, which characterizes the situation, via topological
data, when a collapsing domain with Ricci curvature lower bound is actually an
infranil fiber bundle over an orbifold. In order to apply Theorem~\ref{thm:
main2} for the smoothing purpose, we need to keep track of the distance change
after locally running the Ricci flow. We have the following distance distortion
estimate according to \cite[Lemma 1.11]{HKRX18}:
\begin{lemma}[Distance distortion]\label{lem: dis_dis}
For any $\alpha\in (0,1)$, there is a positive quantity $\Psi_D(\alpha|m)$
with $\lim_{\alpha\to 0}\Psi_D(\alpha|m)=0$, such that under the assumption of
Theorem~\ref{thm: main1}, for any $x,y\in B_g(p,2)$ and any $t\in
(0,\varepsilon^2_P(m,\alpha)]$, if $d_g(x,y)\le \sqrt{t}$, then we have
\begin{align}\label{eqn: distance_distortion}
\left|d_{g(t)}(x,y)- d_g(x,y)\right|\ \le\ \Psi_D(\alpha|m)\sqrt{t}.
\end{align}
\end{lemma}
\begin{proof}
As in the proof of Theorem~\ref{thm: main1}, we consider the
universal covering $\pi:X\to B_g(p,6)$ and we have a Ricci flow solution
$\tilde{h}(t)$ on $X$, with $\tilde{h}(0)=\pi^{\ast}h$ and $\tilde{h}(0)\equiv
\pi^{\ast}g$ on $B_g(p,4)$. Notice that for each $t\in[0,\varepsilon_P^2]$, the
fundamental group $\Gamma:=\pi_1(B_g(p,6))$ acts on $(X,\tilde{h}(t))$ by
discrete isometries and the Ricci flow $g(t)$ on $B_g(p,4)$ is the quotient flow
$(\pi^{-1}(B_g(p,4)),\tilde{h}(t))\slash \Gamma$.

Recall that $(X,\pi^{\ast}h)$ satisfies the assumption of
Theorem~\ref{thm: CTY11} at every point, then we could apply \cite[Lemma
1.11]{HKRX18} to find some positive quantity $\Psi_D(\alpha|m)$
with $\lim_{\alpha\to 0}\Psi_D(\alpha|m)=0$, such that for any
$\tilde{x},\tilde{y}\in X$ with $d_{\pi^{\ast}g}(\tilde{x},\tilde{x})\le
\sqrt{t}$, we have the estimate
\begin{align*}
\left|d_{\tilde{h}(t)}(\tilde{x},\tilde{y})
-d_{\pi^{\ast}g}(\tilde{x},\tilde{y})\right|\ \le\ \Psi_D(\alpha|m)\sqrt{t}.
\end{align*}

By the discussion above, the isometric action of $\Gamma$ on
$(\pi^{-1}(B_g(p,4)),\tilde{h}(t))$ for any $t\in [0,\varepsilon_P^2]$ ensure
that the above distance comparison descends to $B_g(p,2)$, since any minimal
geodesic $\gamma$ realizing the ($g$- or $g(t)$-)distance between two points in
$B_g(p,2)$ lie entirely within $B_g(p,4)$, and is lifted to a minimal 
($\pi^{\ast}g$- or $\tilde{h}(t)$-)geodesic $\tilde{\gamma}$ with the
same length. This gives the desired distance distortion estimate (\ref{eqn:
distance_distortion}).
\end{proof}

We are now ready to prove the infranil fiber bundle theorem:
 \begin{proof}[Proof of Theorem~\ref{thm: main3}] 
 Let us recall that in \cite[Theorem 2.1]{CFG92} there is dimensional constant
 $\delta(m)>0$ which we denote $\lambda_{CFG}$. Let us fix the largest possible
 $\alpha_F\in (0,10^{-1})$ so that $\Psi_D(\alpha_F|m)\le
 10^{-2}\lambda_{CFG}^2$, and we can consider $\alpha_F$ as a constant only
 determined by $m$. Now there are constants $\delta_{O}(\alpha_F)>0$ and
 $\varepsilon_{O}(\alpha_F)>0$ obtained from Theorem~\ref{thm: CTY11}. Now
 we put
\begin{align*} 
\delta_F(m,l,\bar{\iota},\alpha_F)\ :=\
10^{-2}\min\left\{\varepsilon_{O}(m,l,\bar{\iota},\alpha_F)
\lambda_{CFG}^2,\ \delta_{O}(m,l,\bar{\iota},\alpha_F)\right\}.
\end{align*}

For any $\delta<\delta_F$ and any $p\in B_g(K,4\bar{\iota})$, we have 
\begin{align*}
d_{GH}\left(B_g(p,\bar{\iota}), \mathbb{B}^k(\bar{\iota})\slash
\Gamma_{\Phi(p)}\right)\ \le\ &d_{GH}\left(B_g(p,\bar{\iota}),
B_{d_Z}(\Phi(p),\bar{\iota}) \right)+ d_{GH}\left(B_{d_Z}(\Phi(p), \bar{\iota}),
\mathbb{B}^k(\bar{\iota})\slash \Gamma_{\Phi(p)} \right)\\
<\ &\delta_O,
\end{align*}
and Theorem~\ref{thm: main2} applies to start the Ricci flow over $B_g(K,
\bar{\iota})$ for a period no shorter than $T_F:=\varepsilon_O^{2}$, satisfying 
the curvature bound
\begin{align*}
\sup_{B_g(K,\bar{\iota})}\left|\Rm_{g(T_F)}\right|_{g(T_F)}\ \le\ 2T_F^{-1}.
\end{align*} 
Moreover, Shi's estimate in \cite{Shi} provides finite constants $C_n$ ($n\ge
1$) such that
\begin{align*}
\sup_{B_g(K,\bar{\iota})}\left|\nabla^n\Rm_{g(T_F)}\right|_{g(T_F)}\ \le\
C_n(m,T_F,\alpha_F).
\end{align*} 
The rescaled metric $g(T_F)':=\min\left\{2T_F,\bar{\iota}^2\right\}^{-1}g(T_F)$
then satisfies, with $C_0'=1$, the regularity estimates
\begin{align}\label{eqn: regularity_gbar}
\forall n\in
\mathbb{N},\quad
\sup_{B_g(K,\bar{\iota})}\left|\nabla^n\Rm_{g(T_F)'}\right|_{g(T_F)'}\ \le\
C'_n(m,T_F,\alpha_F).
\end{align}
By the scaling invariant estimate in Lemma~\ref{lem: dis_dis}, we know that the
identity map on $B_g(K,\bar{\iota})$ provides a
$(1,\Psi_D(\alpha_F))$-Gromov-Hausdorff approximation between the metrics
$(2T_F)^{-1}g$ and $g(T_F)'$ --- here for any $\varepsilon >0$, we define a
$(1,\varepsilon)$-\emph{Gromov-Hausdorff approximation} between metric spaces
$(A,d_A)$ and $(B,d_B)$ as an $\varepsilon$-dense map $\varphi: A\to B$ such
that $\forall a\in A$, $\varphi|_{B_{d_A}(a,1)}:B_{d_A}(a,1)\to
B_{d_B}(\varphi(a),1+\varepsilon)$ is an $\varepsilon$-Gromov-Hausdorff
approximation; see also \cite[\S 2]{HKRX18} for the definition. 

Therefore, the region $(B_{g}(K,\bar{\iota}),g(T_F)')$ collapses on unit scale
to a $(k,l,1)$-controlled orbifold with regular metric (\ref{eqn:
regularity_gbar}), and combining the localization and the frame bundle argument
(see \cite{Fukaya89, CFG92}), it is easily seen that $K$ has an open
neighborhood $U\Subset B_g(K,\bar{\iota})$ which fibers over $Z_{4\bar{\iota}}$
by infranil fibers. In fact, by \cite[Theorem 1.7]{CFG92}, there is a nilpotent
Killing structure over $U$. Here we notice that the only technical alternation
is replacing \cite[Theorem 2.6]{CFG92} with \cite[Theorem 2.2]{HKRX18}, which
says that the $(1,\Psi_D)$-Gromov-Hausdorff collapsing sufficies to produce the
same conclusions of \cite[Theorem 2.6]{CFG92}.
\end{proof}

\begin{remark}\label{rmk: Reifenberg}
In \cite{Huang20}, Huang proved a fiber bundle theorem for manifolds with
uniform Ricci curvature and local rewinding volume lower bounds, that collapse
to lower dimensional \emph{manifolds} with bounded geometry; see \cite[Theorem
1.3]{Huang20}. The proof relies on applying the canonical Reifenberg method,
developed in \cite{ChCoI, CJN18}, to the non-collapsing local universal
covering space. 
If the collapsing limit is instead a \emph{singular orbifold}, as in the setting
of Theorem~\ref{thm: main3}, an attempt to apply the canonical Reifenberg method
would have to be performed on the finite covering space $Y$ of $B_{g'}(p,25)$,
which has a geodesic $10$-ball Gromov-Hausdorff close to $\mathbb{B}^k(10)$ ---
see the proof of Theorem~\ref{thm: main2} in \S 3. However, the harmonic almost
splitting map obtained from the canonical Reifenberg method is \emph{not}
necessarily equivariant (with respect to the $H$ action on $Y$ and the
$G\cong H$ action on $\mathbb{B}^k(20)$) --- it is not obvious that such a
harmonic map defines a topological fiber bundle structure on
$B_{g'}(p,10)$ over the orbifold neighborhood. On the contrary, the Ricci flow
is a natural regularization that respects isometric group actions,  
particularly adaptive to the collapsing setting. 
While it is possible that an equivariant Reifenberg method appears in future
works, the Ricci flow smoothing method (Theorems~\ref{thm:  main1} and
\ref{thm: main2}) remains necessary at the current stage.
\end{remark}

\begin{remark}\label{rmk: HR20}
We also point out that when $Z$ is (a bounded open subset of) a
$k$-dimensional manifold with locally bounded geometry, 
Theorem~\ref{thm: main3} provides a localization of \cite[Theorem B]{HKRX18}:
see \cite[Remark 2]{HRW20}. But our emphasis here is the extra orbifold
singularity which may very well occur in the collapsing geometry; see
Remark~\ref{rmk: general_orbifold}.
\end{remark}

In fact, if $f:B_g(p,20)\to \mathbb{B}^k(20)\slash G_O$ is an infranil
fiber bundle over an orbifold neighborhood with $G_O<O(k)$,
$\left|G_O\right|<\infty$, and $\diam_g(f^{-1}(z))
<\left|G_O\right|^{-1}\delta$ for any $z\in \mathbb{B}^k(20)\slash
G_O$, we also have the structure described by the assumptions of
Theorem~\ref{thm: main3} --- this presents a local inverse to this theorem. For
more detailed descriptions about the infranil fiber bundle structure over an
orbifold chart, as well as the locally unwarpped neighborhood, we refer the
readers to \cite[\S 7]{Fukaya89} and \cite[\S 3]{Foxy1908}.

We will follow the notations in \cite[Definition]{Fukaya89}. By the fact that
both $\mathbb{B}^k(20)\times F$ and $B_g(p,20)$ are smooth manifolds, it
implies that the action of the finite orbifold group $G_O$ is free of fixed
points, implying that $\theta: G_O \to Aff(F)$ is injective. On the other
hand, let $\pi:X\to \mathbb{B}^k(20)\times F$ denote the universal covering.
We clearly see that $X\cong \mathbb{B}^k(20)\times \tilde{F}$, with $\tilde{F}$
denoting the universal covering of $F$. Consequently,
$\pi_1(\mathbb{B}^k(20)\times F)=\pi_1(F)$. Notice that the map $q\circ
\pi:X\to B_g(p,20)$, with $q: \mathbb{B}^k(20)\times F\to B_g(p,20)$ denoting
the quotient map of the $G_O$ action, is actually a covering map, thanks to
the discrete and free action of $G_O$. Then the simple connectedness of $X$
makes it the universal covering space of $B_g(p,20)$. Since $\pi$ is a normal
covering, we have $\pi_1(F)\triangleleft \pi_1(B_g(p,20))$, with the quotient
$H:=\pi_1(B_g(p,20))\slash \pi_1(F)$ acting on $\mathbb{B}^k(20)\times F$ so
that $(\mathbb{B}^k(20)\times F)\slash H=B_g(p,20)$. By the correspondence
between the covering spaces of $B_g(p,20)$ and subgroups of $\pi_1(B_g(p,20))$,
we know that $H\cong G_O$, whence a surjective group homomorphism
\begin{align*}
\pi_1(B_g(p,20))\ \twoheadrightarrow\ \pi_1(B_g(p,20))\slash
\pi_1(F)\ \cong\ G_O.
\end{align*} 
To see that $\rank\ \tilde{\Gamma}_{\delta}(p)=m-k$, we notice
that  $\diam_gq^{-1}(f^{-1}(Z))\le \delta$ for all
$z\in \mathbb{B}^k(20)\slash G_O$, and thus
$\tilde{\Gamma}_{\delta}(\pi(\tilde{p})) =\pi_1(\mathbb{B}^k(20)\times
F)=\pi_1(F)$, for any $\tilde{p}\in \pi^{-1}(p) \subset X$. Clearly, $\pi_1(F)$
is almost nilpotent with $\rank\ \pi_1(F)= \dim F=m-k$. The following splitting
short exact sequence of groups
\begin{align*}
0\longrightarrow \pi_1(F)\longrightarrow
\pi_1(B_g(p,20))\stackrel{\theta}{\longleftrightarrows} 
G_O\longrightarrow 0,
\end{align*}
then tells that
$\tilde{\Gamma}_{\delta}(p)=\tilde{\Gamma}_{\delta}(\pi(\tilde{p}))\rtimes
G_O$, implying that $\rank\ \tilde{\Gamma}_{\delta}(p)=\rank\
\pi_1(F)=m-k$.

\appendix
\section{Local distance distortion estimates for Ricci flows with collapsing
initial data}

The distance distortion estimates along Ricci flows is a crucial issue in
view of its many natural applications (besides Lemma~\ref{lem: dis_dis} in the
proof of Theorem~\ref{thm: main3}, see also e.g. \cite{ChenWang19} for a
survey). In \cite{Foxy1808} a uniform distance distortion estimate for Ricci
flows with collapsing initial data has been obtained. That estimate is for
compact Ricci flow solutions and only compares the distance functions on nearby
\emph{positive} time slices. Here we present a new distance distortion estimate
which can be seen both as an extension (to $t=0$) and a localization of the
estimate in \cite[Theorem 1.1]{Foxy1808} under some extra assumption on the
Ricci curvature:
\begin{theorem}\label{thm: dis_dis_2}
Given a positive integer $m$, positive constants $\bar{C}_0$, $C_R$, $T\le 1$
and $\alpha\in (0,\frac{1}{2(m-1)})$, there are constants 
$C_D(\bar{C}_0,C_R,m)\ge 1$ and $T_D(\bar{C}_0,C_R,m)\in (0,T]$ such that for an
$m$-dimensional complete Ricci flow $(M,g(t))$ defined for $t\in [0,T]$, if
for some $x_0\in M$ and any $t\in [0,T]$ we have
\begin{align}
\Sc_{g(0)}\ &\ge\ -C_R\ \text{in}\ B_{g(0)}(x_0,10),
\label{eqn: key_Sc}\\
\left|\Rc_{g(t)}\right|_{g(t)}\ &\le\ 2(m-1)\alpha t^{-1}\ \text{in}\ 
B_{g(t)}\left(x_0,10+\sqrt{t}\right),
\label{eqn: key_spacetime_Rc}
\end{align}
and the initial metric has a uniform bound $\bar{C}_0$ on the doubling and
Poincar\'e constant for the geodesic ball $B_{g(0)}(x_0,10)$, then
for any $x,y\in B_{g(0)}(x_0,\sqrt{T_D})$ and $t\in [0,T_D]$, we
have
\begin{align}
C_D^{-1}d_{g(0)}(x,y)^{1+4(m-1)\alpha}d_{g(0)}(x,y)\ \le\ d_{g(t)}(x,y)\ \le\
C_Dd_{g(0)}(x,y)^{1-4(m-1)\alpha}.
\end{align}
\end{theorem}

The proof of the lemma is based on the theory of local entropy \cite{Wang18},
as well as the local entropy lower bound by volume ratio obtained in
\cite{Foxy1808}. It is inspired by the corresponding results for
non-collapsing initial data in \cite{HKRX18}, and relies on a ball
containment argument as surveyed in \cite[\S 3]{ChenWang19}. 

\begin{remark}\label{rmk: SC92}
  Here we will rely on \cite[Theorem 2.1]{Saloff-Coste92}, and we point out
  that the assumptions \cite[(1) and (2)]{Saloff-Coste92} are only needed
  locally, i.e. the same conclusion of \cite[Theorem 2.1]{Saloff-Coste92} holds
  even if these conditions on the doubling and Poincar\'e constants are assumed
  only within the geodesic ball $B_{g(0)}\left(x_0,10\right)$: notice that the
  only global result needed in the proof of this theorem is \cite[Theorem
 2.2]{Saloff-Coste92}, in which the constants involved only depend on the
 dimension of the manifold; whereas all other arguments leading to the
 application of \cite[Theorem 2.2]{Saloff-Coste92} only depend on the doubling
 and Poincar\'e constants within the geodesic ball in question.
\end{remark}

\begin{proof}
By taking $r_0=\sqrt{t}$ in \cite[Lemma 8.3]{Perelman} (see also
\cite[\S 17]{Hamilton93}), the Ricci curvature upper bound in (\ref{eqn:
key_spacetime_Rc}) tells that if $x,y\in B_{g(t)}\left(x_0,10\right)$ whenever 
$t\le T$, then
\begin{align}\label{eqn: d0_le_dt}
d_{g(0)}(x,y)\ \le\ d_{g(t)}(x,y)+3(m-1)\sqrt{t}. 
\end{align} 
 To control the distance expansion, we first establish the following \emph{a
 priori} estimate 
\begin{claim}\label{clm: dis_dis_claim}
There exists a uniform constant $\bar{C}_1(\bar{C}_0,C_R, m)>0$ such that if
$x,y \in  B_{g(t)}(x_0,10)$ for any $t\le
t_0:=d_0^2$ with $d_0:=d_{g(0)}(x,y)\le \sqrt{T}$, then
$d_{g(t_0)}(x,y)\le \bar{C}_1d_0$.
\end{claim}
\begin{proof}[Proof of the claim]
To see this, we let $\sigma:[0,1]\to B_{g(0)}\left(x,d_0\right)$ be a minimal
$g(0)$-geodesic realizing $d_0$, and without loss of generality we may assume
that $d_{g(t_0)}(x,y)\ge 10d_0$. Now let $\{\sigma(s_i)\}_{i=1}^N$ be a
maximal collection of points on the image of $\sigma$ so that
$d_{g(t_0)}(\sigma(s_i),\sigma(s_{j}))\ge 2d_0$ when $i\not=j$. By the
maximality of $\left\{\sigma(s_i)\right\}$, we see that
$B_{g(t_0)}(\sigma(s_i), d_0)\cap B_{g(t_0)}(\sigma(s_j),d_0)=\emptyset$, and
it is also clear that $\left\{B_{g(t_0)}(\sigma(s_i),2d_0)\right\}$ cover
$Image(\sigma)$.
Especially, we have
\begin{align*}
d_{g(t_0)}(x,y)\le \sum_{i=1}^{N-1}d_{g(t_0)}(\sigma(s_i),
\sigma(s_{i+1}))\ \le\ 4Nd_0.
\end{align*} 
We therefore only need to bound $N$ uniformly from above. 
Moreover, by (\ref{eqn: d0_le_dt}) we see that for each $i=1,\ldots,N$,
\begin{align*}
\forall y'\in B_{g(t_0)}(\sigma(s_i),d_0),\quad d_{g(0)}(\sigma(s_i),y')\ \le\
&d_{g(t_0)}(\sigma(s_i),y')+8(m-1)\sqrt{\alpha t_0}\\
\le\ &(1+8(m-1)\sqrt{\alpha})d_0.
\end{align*}
Therefore, it is easily seen that each
$B_{g(t_0)}(\sigma(s_i),d_0)\subset
B_{g(0)}(\sigma(s_i),(1+8(m-1)\sqrt{\alpha})d_0)$, and consequently,
as each $\sigma(s_i)\in B_{g(0)}(x,d_0)$, we have
\begin{align}
\label{eqn: ball_containment}
\bigcup_{i=0}^kB_{g(t_0)}(\sigma(s_i),d_0)\ \subset\ B_{g(0)}(x,
(2+8(m-1)\sqrt{\alpha})d_0).
\end{align}

We now study the local entropy associated to the various metric balls. We first
recall that the uniform bound $\bar{C}_0$ on the doubling and Poincar\'e constants
gives a uniform bound on the Sobolev constant $C_S=C_S(m,\bar{C}_0)$, according
to \cite[Theorem 2.1]{Saloff-Coste92} (see also Remark~\ref{rmk: SC92} and 
\cite[Proposition 2.1]{Foxy1808}). Therefore, following the argument in
\cite[\S 3.1]{Foxy1808}, we get to \cite[(3.1)]{Foxy1808}, and by (\ref{eqn:
key_Sc}) we have
\begin{align*}
&\mathcal{W}\left(B_{g(0)}(x,8md_0),g(0),v^2,\tau\right)\\
\ge\ &\log \left|B_{g(0)}(x,8md_0)\right|d_0^{-m}
-\left(C_R+(8md_0)^{-2}\right)\tau-\frac{m}{2}\log
\left(512C_Sem^3\pi\right).
\end{align*}
for any $\tau>0$ and any $v\in W^{1,2}_0(B_{g(0)}(x,8md_0))$ with
$\int_{M}v^2\dvol_g=1$. Here due to the selection of $v$, we have the Perelman's
$\mathcal{W}$-entropy equal to the \emph{local entropy}
$\mathcal{W}\left(B_{g(0)}(x,8md_0),g(0),v^2,\tau\right)$, as defined in
\cite[\S 2]{Wang18}. Taking the infimum of $\mathcal{W}
\left(B_{g(0)}(x,8md_0),g(0),v^2,\tau\right)$ among all admissible $v$
described above, we see that for any $\tau>0$, 
\begin{align*}
&\boldsymbol{\mu}\left(B_{g(0)}(x,8md_0),g(0),\tau\right)\\
 \ge\ &\log \left|B_{g(0)}(x,8md_0)\right|d_0^{-m}
 -\left(C_R+(8md_0)^{-2}\right)\tau-\frac{m}{2}\log
 \left(512C_Sem^3\pi\right).
\end{align*}
Consequently, for
$\boldsymbol{\nu}\left(B_{g(0)}(x,8md_0),g(0),\tau\right)=\inf_{s\in
(0,\tau]}\boldsymbol{\mu}\left(B_{g(0)}(x,8md_0),g(0),s\right)$ we have
\begin{align}\label{eqn: nu_8md0_lb}
\begin{split}
&\boldsymbol{\nu}\left(B_{g(0)}(x,8md_0),g(0),\tau\right)\\ 
\ge\ &\log \left|B_{g(0)}(x,8md_0)\right|d_0^{-m}
-\left(C_R+(8md_0)^{-2}\right)\tau-\frac{m}{2}\log
\left(512C_Sem^3\pi\right).
\end{split}
\end{align}
On the other hand, by (\ref{eqn: key_spacetime_Rc}) we have
$\left|\Sc_{g(t)}\right|\le 2m(m-1)\alpha t^{-1}$ for $t>0$, and applying
\cite[Theorem 3.6]{Wang18}, we can bound the local entropy from above by volume
ratio:
\begin{align}\label{eqn: nu_i_t_ub}
\boldsymbol{\nu}\left(B_{g(t)}(\sigma(s_i),d_0),g(t),d_0^2\right)\ \le\ \log
\frac{\left|B_{g(t)}(\sigma(s_i),d_0)\right|}{\omega_md_0^m}
+\left(2^{m+7}+2m(m-1)\alpha t^{-1}d_0^2\right);
\end{align}
while by (\ref{eqn: key_spacetime_Rc}) and the effective monotonicity of the
local entropy (\cite[Theorem 5.4]{Wang18}), we see that
\begin{align}\label{eqn: nu_i_almost_increasing}
\boldsymbol{\nu}\left(B_{g(t)}(\sigma(s_i),d_0),g(t),d_0^2 \right)\ \ge\
\boldsymbol{\nu}\left(B_{g(0)}(\sigma(s_i),3d_0),g(0),d_0^2+t \right)-1.
\end{align}
Moreover, since $B_{g(0)}(\sigma(s_i),3d_0)\subset B_{g(0)}(x,8md_0)$ for each
$i=1,\ldots,N$, applying \cite[Proposition 2.1]{Wang18} and (\ref{eqn:
nu_8md0_lb}) we see that
\begin{align}\label{eqn: nu_0_lb}
\begin{split}
&\boldsymbol{\nu}\left(B_{g(0)}(\sigma(s_i),3d_0),g(0),\tau \right)\\ 
\ge\ &\log \left|B_{g(0)}(x,8md_0)\right|d_0^{-m}
-\left(C_R+(8md_0)^{-2}\right)\tau-\frac{m}{2}\log
\left(512C_Sem^3\pi\right).
\end{split}
\end{align}

We now obtain a uniform bound on $N$. Denoting $\boldsymbol{\nu}_{i,s}(r,\tau)
:=\boldsymbol{\nu} (B_{g(s)}(\sigma(s_i),r),g(s),\tau)$ for each
$i=1,\ldots,N$, $r\le 8md_0$, $s\le t_0$ and $\tau\le d_0^2+t_0$, we
consecutively apply (\ref{eqn: nu_i_t_ub}), (\ref{eqn: nu_i_almost_increasing})
and (\ref{eqn: nu_0_lb}) with $t_0=d_0^2$ to see that 
\begin{align}
\begin{split}
\left|B_{g(t_0)}(\sigma(s_i),d_0)\right|d_0^{-m}\ \ge\
&\omega_me^{\boldsymbol{\nu}_{i,t_0}\left(d_0,d_0^2\right)
-2^{m+7}-2m(m-1)\alpha}\\
\ge\ &\omega_m
e^{\boldsymbol{\nu}_{i,0}\left(3d_0,d_0^2+t_0\right)
-2^{m+7}-2m(m-1)\alpha-1}\\
\ge\ &\frac{\omega_m
e^{-2^{m+7}-2C_Rd_0^2-2m(m-1)\alpha-2}}{\left(512C_S em^3\pi
\right)^{\frac{m}{2}}} \left|B_{g(0)}(x,8md_0)\right|d_0^{-m}.
\end{split}
\end{align}
Therefore, by the mutual disjointness of
$\left\{B_{g(t_0)}(\sigma(s_i),d_0)\right\}$ and (\ref{eqn: ball_containment})
we have
\begin{align}
\begin{split}
\left|B_{g(0)}(x,8md_0)\right|_{g(0)}\ \ge\
&\sum_{i=1}^N\left|B_{g(t)}(\sigma(s_i),d_0)\right|_{g(0)}\\
\ge\ &\sum_{i=1}^N\left|B_{g(t)}(\sigma(s_i),d_0)\right|_{g(t)}\\
\ge\ &\sum_{i=1}^N C\left|B_{g(0)}(x,8md_0)\right|_{g(0)},
\end{split}
\end{align}
where $C=\omega_m
e^{-2^{m+7}-2C_Rd_0^2-2m(m-1)\alpha-2}(512C_S em^3\pi )^{-\frac{m}{2}}$.
Consequently, we easily see that 
\begin{align}
N\ \le\ \omega_m^{-1}e^{2^{m+7}+C_R+2m(m-1) +2} 
(288C_S em\pi)^{\frac{m}{2}}+10\ =:\ \frac{1}{4}\bar{C}_1(\bar{C}_0,C_R,m),
\end{align}
since $\alpha<1$, and we emphasize that $C_S$ is solely determined by
$\bar{C}_0$ and $m$.
\end{proof}
Now let $T_0\le T$ be the first time when some point $y\in
B_{g(0)}(x_0,\sqrt{T_0})$ sees $d_{g(T_0)}(x_0,y)\ge 10$, then applying
the estimate established by the claim with $x_0,y$, we have  
$T_0\ \ge\ 100\bar{C}_1^{-2}=:T_D(\bar{C}_0,C_R,m)$. Especially, for any
$x,y\in B_{g(0)}(x_0,\sqrt{T_D})$, we have (\ref{eqn: d0_le_dt}) holds
without any extra assumption, and for any such points,
\begin{align}\label{eqn: distance_ratio}
\forall 0<s<t\le T_D,\quad \left( \frac{s}{t}\right)^{4(m-1)\alpha}\ \le\
\frac{d_{g(t)}(x,y)}{d_{g(s)}(x,y)}\ \le\
\left(\frac{t}{s}\right)^{4(m-1)\alpha}.
 \end{align} 
 Therefore, reasoning in the same way as in \cite[Appendix]{HKRX18}, we have
\begin{align}\label{eqn: dis_dis_upper}
\forall x,y\in B_{g(0)}(x_0,\sqrt{T_D}),\quad d_{g(0)}(x,y) \le\ 
3m d_{g(t)}(x,y)^{\frac{1}{4(m-1)\alpha+1}}.
 \end{align} 

We now show the other side of the estimate following the same argument as in
\cite[Appendix]{HKRX18}. Given $x,y\in B_{g(0)}(x_0,\sqrt{T_D})$ satisfying
$d_0=d_{g(0)}(x,y)\le \sqrt{T_D}$ and given $t\le T_D$, we begin with setting 
$N:= \left\lceil d_0t^{-\frac{1}{2}}\right\rceil+1$ so that $d_0<N\sqrt{t} \le
2\sqrt{T}\le 2$. Dividing a minimal $g(0)$-geodesic $\sigma$ that realizes $d_0$
into $N$ pieces of equal length, i.e. $\left|\sigma|_{[s_i,s_{i+1}]}\right|
=N^{-1}d_0$ for $0=s_0<\ldots<s_N=1$, by the claim above we see that
\begin{align*}
\forall i=0,1,\ldots,N-1,\quad
d_{g(N^{-2}d_0^2)}(\sigma(s_i),\sigma(s_{i+1}))\ \le\ \bar{C}_1N^{-1}d_0;
\end{align*}
moerover, since $t>N^{-2}d_0^2$, by (\ref{eqn: distance_ratio}) we have 
\begin{align*}
d_{g(N^{-2}d_0^2)}(x,y)\ \ge\ \left(N^{-2}d_0^{2}t^{-1}\right)^{2(m-1)\alpha}
d_{g(t)}(x,y).
\end{align*} 
Now adding through $i=0,1,\ldots,N-1$, by these inequalities we have 
\begin{align*}
d_0\ =\ &\sum_{i=0}^{N-1}\left|\sigma|_{[s_i,s_{i+1}]}\right|\\
\ge\ &\bar{C}_1^{-1}\sum_{i=0}^{N-1}d_{g(N^{-2}d_0^2)}
(\sigma(s_i),\sigma(s_{i+1}))\\
\ge\ &\bar{C}_1^{-1}d_{g(N^{-2}d_0)}(\sigma(0),\sigma(1))\\
\ge\ &\bar{C}_1^{-1}\left(N^{-2}d_0^2t^{-1} \right)^{2(m-1)\alpha}
d_{g(t)}(x,y)\\
\ge\ &\bar{C}_1^{-1}16^{-(m-1)\alpha}d_0^{4(m-1)\alpha} d_{g(t)}(x,y),
\end{align*} 
and consequently, we have $d_{g(t)}(x,y)\le 16^m\bar{C}_1d_0^{1-4(m-1)\alpha}$.

Combining this with (\ref{eqn: dis_dis_upper}) we have the desired estimate
\begin{align*}
\forall x,y\in B_{g(0)}(x_0,\sqrt{T_D}),\ \forall t\le T_D,\quad  
C_D^{-1}d_{g(0)}(x,y)^{1+4(m-1)\alpha}\ \le\ d_{g(t)}(x,y)\ \le\
C_Dd_{g(0)}(x,y)^{1-4(m-1)\alpha},
\end{align*}
where $C_D:=\max\left\{(8m)^{4m},16^m\bar{C}_1\right\}$, only depending on
$\bar{C}_0$ and $C_R$.
\end{proof}

Arguing in the same way, if (\ref{eqn: key_spacetime_Rc}) can be assumed
globally on $M$, then we have the following 
\begin{corollary}\label{cor: dis_dis}
With the same assumptions as in Theorem~\ref{thm: dis_dis_2}, but with
(\ref{eqn: key_spacetime_Rc}) replaced by 
\begin{align}\label{eqn: dis_dis_global_Rc}
\forall t\le T,\quad \sup_M\left|\Rc_{g(t)}\right|_{g(t)}\ \le\ 2(m-1)\alpha
t^{-1},
\end{align}
then we have for any $t\le T$ and $x,y\in B_{g(0)}(x_0,5)$ satisfying
$d_{g(0)}(x,y)\le 1$,
\begin{align}\label{eqn: dis_dis_estimate}
C_D^{-1}d_{g(0)}(x,y)^{1+4(m-1)\alpha}\ \le\ d_{g(t)}(x,y)\ \le\
C_Dd_{g(0)}(x,y)^{1-4(m-1)\alpha}.
\end{align}
\end{corollary}
\begin{proof}
By (\ref{eqn: dis_dis_global_Rc}), we see that (\ref{eqn: d0_le_dt}) holds for
any $x,y\in B_{g(0)}(x_0,5)$, without any extra assumption, and thus
Claim~\ref{clm: dis_dis_claim} holds for all such points. Therefore, the rest of
the arguments follow without needing to confine ourselves in a smaller geodesic
ball.
\end{proof}

\subsection*{Acknowledgements.} Both authors would like to thank Professor
Xiaochun Rong for enlightening discussions and his warm encouragement. The
second-named author is partially supported by the General Program of the
National Natural Science Foundation of China (Grant No. 11971452) and the
research fund of USTC.

\vspace{0.7 in}


\begin{thebibliography}{100}


\bibitem{AGS14} Luigi Ambrosio, Nicola Gigli and Giuseppe Savar\'e, Metric
measure spaces with Riemannian Ricci curvature bounded below. \emph{Duke Math.
J.} 163 (2014), no. 7, 1405-1490.



\bibitem{Bamler16} Richard Bamler, A ricci flow proof of a result by Gromov on
lower bounds for scalar curvature. \emph{Math. Res. Lett.}
23 (2016), no. 2, 325-337.


\bibitem{BKN89} Shigetoshi Bando, Atsushi Kasue and Hiraku Nakajima, On a
construction of coordinates at infinity on manifolds with fast curvature decay
and maximal volume growth. \emph{Invent. Math.} 97 (1989), no. 2, 313-349.






\bibitem{CM18} Fabio Cavalletti and Andrea Mondino, Almost Euclidean
isoperimetric inequalities in spaces satisfying local Ricci curvature lower
bounds. \emph{Int. Math. Res. Not.} 2018, doi:10.1093\slash imrn\slash rny070.

\bibitem{CTY11} Albert Chau, Luen-Fai Tam and Chengjie Yu, Pseudolocality for
the Ricci flow and applications. \emph{Canad. J. Math.} 63 (2011), no. 1, 55-85.

\bibitem{ChCo0} Jeff Cheeger and Tobias Colding, Lower bounds on Ricci curvature
and the almost rigidity of warped products. \emph{Ann. of Math.} 144 (1996),
no. 1, 189-237.

\bibitem{ChCoI} Jeff Cheeger and Tobias Colding, On the structure of spaces
with Ricci curvature bounded below. I. \emph{J. Differential Geom.} 46 (1997),
no. 3, 406-480.


\bibitem{ChCoIII} Jeff Cheeger and Tobias Colding, On the structure of spaces
with Ricci curvature bounded below. III, \emph{J. Differential Geom.} 54
 (2000), 37-74.








\bibitem{CFG92} Jeff Cheeger, Kenji Fukaya and Mikhail Gromov, Nilpotent
structures and invariant metrics on collapsed manifolds. \emph{J. Amer.
Math. Soc.} 5 (1992), no. 2, 327-372.

\bibitem{CJN18} Jeff Cheeger, Wenshuai Jiang and Aaron Naber, Rectifiability of
singular sets in noncollapsed spaces with Ricci curvature bounded below.
\emph{Preprint}, arXiv: 1805.07988.




\bibitem{ChenWang19} Xiuxiong Chen and Bing Wang, Remarkes of
weak-compactness along K\"ahler Ricci flow. \emph{Proceedings of the Seventh
International Congress of Chinese Mathematicians. Vol. II,} 203-33, Adv. Lect.
Math. (ALM) 44, \emph{Int. Press, Somerville, MA,} 2019.


\bibitem{Colding97} Tobias H. Colding, Ricci curvature and volume convergence.
\emph{Ann. of Math. (2)} 145 (1997), no. 3, 477-501. 

\bibitem{ColdingNaber} Tobias H. Colding and Aaron C. Naber, Sharp H\"older
continuity of tangent cones for spaces with a lower Ricci curvature bound and
applications. \emph{Ann. of Math. (2)} 176 (2012), no. 2, 1173-1229.

\bibitem{DWY96} Xianzhe Dai, Guofang Wei and Rugang Ye, Smoothing Riemannian
metrics with Ricci curvature bounds. \emph{Manuscripta Math.} 90 (1996), no. 1, 
49-61.

\bibitem{DPG18} Guido De Philippis and Nicola Gigli, Non-collapsed spaces with
Ricci curvature bounded below. \emph{J. \'Ec. polytech. Math} 5 (2018), 613-650.



  

 \bibitem{Fukaya88} Kenji Fukaya, A boundary of the set of the Riemannian
manifolds with bounded curvatures and diameters. \emph{J. Differential Geom.}
 28 (1988), 1-21.

\bibitem{Fukaya89} Kenji Fukaya, Collapsing Riemannian manifolds to ones with
lower dimension II. \emph{J. Math. Soc. Japan} 41 (1989), no. 2, 333-356.

\bibitem{FY92} Kenji Fukaya and Takao Yamaguchi, The fundamental groups of
almost non-negatively curved manifolds. \emph{Ann. of Math. (2)} 136 (1992),
no. 2, 253-333.


\bibitem{Gigli18} Nicola Gigli, Lecture notes on differential calculus on $RCD$
spaces. \emph{Publ. Res. Inst. Math. Sci.} 54 (2018), no. 4, 855-918.

\bibitem{Gromov78b} Mikhail Gromov, Almost flat manifolds. \emph{J.
Differential Geom.} 13 (1978), no. 2, 231-241.

\bibitem{Gromov} Mikhail Gromov, Paul Levi's isoperimetric inequality.
\emph{Preprint}, IH\'ES.

\bibitem{Gromov18} Mikhail Gromov, Metric inequalities with scalar curvature.
\emph{Geom. Funct. Anal.} 28 (2018), no. 3, 645-726.






\bibitem{Hamilton} Richard Hamilton, Three-manifolds with positive Ricci
curvature. \emph{J. Differential Geom.} 17 (1982), no. 2, 255-306.

\bibitem{Hamilton93} Richard Hamilton, The formation of singularities in Ricci
flow. \emph{Surveys in differential geometry, Vol. II (Cambridge, MA, 1993),}
7-136, \emph{Int. Press, Cambridge, MA,} 1995.

\bibitem{He16} Fei He, Existence and applications of Ricci flows via
pseudolocality. \emph{Preprint}, arXiv: 1610.01735.


\bibitem{Hochard} Raphael Hochard, Short-time existence of the Ricci flow on
complete, non-collapsed $3$-manifolds with Ricci curvature bounded from below.
\emph{Preprint}, arXiv: 1603.08726.

\bibitem{Honda20} Shouhei Honda, Collapsed Ricci limit spaces as non-collapsed
RCD spaces. \emph{SIGMA Symmetry Integrability Geom. Methods Appl.} 16 (2020),
Paper No. 021, 10 pp.

\bibitem{Huang20} Hongzhi Huang, Fibrations and stability of compact group
actions on manifolds with local bounded Ricci covering geometry. \emph{Front.
Math. China} 15 (2020), no. 1, 69-89.

\bibitem{HKRX18} Hongzhi Huang, Lingling Kong, Xiaochun Rong and Shicheng Xu,
Collapsed manifolds with Ricci bounded covering geometry. \emph{Preprint},
arXiv: 1808.03774, to appear in \emph{Trans. Amer. Math. Soc.} DOI:
https:\slash \slash doi.org\slash 10.1090\slash tran\slash 8177.



\bibitem{Foxy1808} Shaosai Huang, Notes on Ricci flows with collapsing initial
data (I): Distance distortion. \emph{Trans. Amer. Math. Soc.} 373 (2020), no.
6, 4389-4144.

\bibitem{Foxy1908} Shaosai Huang, On the long-time behavior of immortal Ricci
flows. \emph{Preprint}, arXiv: 1908.05410.



\bibitem{HRW20} Shaosai Huang, Xiaochun Rong and Bing Wang, Collapsing
geometry with Ricci curvature bounded below and Ricci flow smoothing.
\emph{Preprint}, arXiv: 2008.12419.

\bibitem{HW20a} Shaosai Huang and Bing Wang, Rigidity of the first Betti
 number via Ricci flow smoothing. \emph{Preprint}, arXiv: 2004.09762. 
 

\bibitem{KPT10} Vitali Kapovitch, Anton Petrunin and Wilderich Tuschmann,
Nilpotency, almost nonnegative curvature and the gradient flow on Alexandrov
spaces. \emph{Ann. of Math. (2)} 171 (2010), no. 1, 343-373.

\bibitem{KW11} Vitali Kapovitch and Burkhard Wilking, Structure of fundamental
groups of manifolds with Ricci curvature bounded below. \emph{Preprint}, arXiv:
1105.5955.


\bibitem{Lai19} Yi Lai, Ricci flow under local almost non-negative curvature
conditions. \emph{Adv. Math.} 343 (2019), 353-392.






\bibitem{LiYau} Peter Li and Shing-Tung Yau, On the parabolic kernel of the
Schr\"odinger operator. \emph{Acta Math.} 156 (1986), no. 3-4, 153-201.


\bibitem{LS18} Gang Liu and G\'abor Sz\'ekelyhidi, Gromov-Hausdorff limits of
K\"ahler manifolds with Ricci curvature bounded below. \emph{Preprint}, arXiv:
1804.08567.





 
\bibitem{NaberZhang} Aaron Naber and Ruobing Zhang, Topology and
$\varepsilon$-regularity theorems on collapsed manifolds with Ricci curvature
bounds. \emph{Geom. Topol.} 20 (2016), no. 5, 2575-2664.

\bibitem{Perelman} Grisha Perelman, The entropy formula for the Ricci flow and
its applications. \emph{Preprint}, arXiv: math\slash 0211159.





\bibitem{Rong20} Xiaochun Rong, A generalized Gromov's theorem on almost flat
manifolds and applications. \emph{In preparation.} 


\bibitem{Ruh} Ernst A. Ruh, Almost flat manifolds. \emph{J. Differential Geom.}
17 (1982), no. 1, 1-14.

\bibitem{Saloff-Coste92} Laurent Saloff-Coste, A note on Poincar\'e, Sobolev,
and Harnack inequalities. \emph{Internat. Math. Res. Notices} 1992, no. 2,
27-38.


\bibitem{YauLecture} Richard Schoen and Shing-Tung Yau, Lectures on differential
geometry. Lecture notes prepared by Wei Yue Ding, Kung Ching Chang [Gong Qing
Zhang], Jia Qing Zhong and Yi Chao Xu. Translated from the Chinese by Ding and
S. Y. Cheng. With a preface translated from the Chinese by Kaising Tso.
Conference Proceedings and Lecture Notes in Geometry and Topology, I.
\emph{International Press, Cambridge, MA,} 1994. v+235 pp. ISBN: 1-57146-012-8

\bibitem{Shi} Wan-Xiong Shi, Deforming the metric on complete Riemannian
manifolds. \emph{J. Differential Geom.} 30 (1989), no. 1, 223-301.






\bibitem{ST17} Miles Simon and Peter Topping, Local mollification of Riemannian
metrics using Ricci flow, and Ricci limit spaces. \emph{Preprint}, arXiv:
1706.09490. To appear in \emph{Geom. Topol.}

\bibitem{Sormani17} Christina Sormani, Scalar curvature and intrinsic flat
convergence. \emph{Measure theory in non-smooth spaces}, 288-338, Partial
Differ. Equ. Meas. Theory, \emph{De Gruyter Open, Warsaw,} 2017.

\bibitem{TianWang} Gang Tian and Bing Wang, On the structure of almost Einstein
manifolds. \emph{J. Amer. Math. Soc.} 28 (2015), no. 4, 1169-1209.

\bibitem{Topping10} Peter Topping, Ricci flow compactness via pseudolocality,
and flows with incomplete metrics. \emph{J. Eur. Math. Soc. (JEMS)} 12 (2010),
no. 6, 1429-1451.

\bibitem{Wang18} Bing Wang, The local entropy along Ricci flow Part A: the
no-local-collapsing theorems. \emph{Camb. J. Math.} 6 (2018), no. 3, 267-346.

\end{thebibliography}
\end{document}